\documentclass{article} 

\usepackage{hyperref}

\usepackage{amsmath}

\usepackage{algorithmic}
\usepackage{algorithm}

\usepackage{amssymb}

\usepackage{amsthm}
\usepackage{mathrsfs}

\usepackage{tikz}
\usetikzlibrary{calc}
\usetikzlibrary{arrows}

\usepackage{graphicx}

\usepackage{appendix}

\usepackage{enumitem} 

\theoremstyle{theorem}
\newtheorem{thm}{Theorem}[section]
\newtheorem{lemma}[thm]{Lemma}
\newtheorem{prop}[thm]{Proposition}
\newtheorem{corol}[thm]{Corollary}

\newtheorem{result}{Result}[section]

\theoremstyle{remark}
\newtheorem{rmk}{Remark}

\theoremstyle{definition}
\newtheorem{defn}{Definition}
\newtheorem{assumption}{Assumption}
\newtheorem{example}{Example}

\usepackage{amsmath}
\usepackage{amssymb}
\usepackage{bbm}

\usepackage[T1]{fontenc}

\usepackage{color}
\usepackage{soul}

\numberwithin{equation}{section}

\newcommand\omicron{o}
\newcommand{\1}{\mathbbm{1}}
\newcommand{\E}{\mathbb{E}}
\renewcommand{\P}{\mathbb{P}}
\newcommand{\Q}{\mathbb{Q}}

\makeatletter
\newcommand{\subalign}[1]{%
  \vcenter{%
    \Let@ \restore@math@cr \default@tag
    \baselineskip\fontdimen10 \scriptfont\tw@
    \advance\baselineskip\fontdimen12 \scriptfont\tw@
    \lineskip\thr@@\fontdimen8 \scriptfont\thr@@
    \lineskiplimit\lineskip
    \ialign{\hfil$\m@th\scriptstyle##$&$\m@th\scriptstyle{}##$\crcr
      #1\crcr
    }%
  }
}
\makeatother
\DeclareMathOperator*{\argmin}{arg\,min}

\title{Efficient Rare-Event Simulation for Multiple Jump Events in Regularly Varying Random Walks and Compound Poisson Processes}
\author{
        By Bohan Chen, Jose Blanchet, Chang-Han Rhee \& Bert Zwart \\\\
        Centrum  Wiskunde \& Informatica and Columbia University
}
\date{\today}

\begin{document}
\maketitle

\begin{abstract}
We propose a class of strongly efficient rare event simulation estimators for random walks and compound Poisson processes with a regularly varying increment/jump-size distribution in a general large deviations regime. 
Our estimator is based on an importance sampling strategy that hinges on the heavy-tailed sample path large deviations result recently established in \cite{rheeblanchetzwart2016}. 
The new estimators are straightforward to implement and can be used to systematically evaluate the probability of a wide range of rare events with bounded relative error. 
They are \textquotedblleft universal\textquotedblright\ in the sense that a single importance sampling scheme applies to a very general class of rare events that arise in heavy-tailed systems. 
In particular, our estimators can deal with rare events that are caused by multiple big jumps (therefore, beyond the usual principle of a single big jump) as well as multidimensional processes such as the buffer content process of a queueing network. 
We illustrate the versatility of our approach with several applications that arise in the context of mathematical finance, actuarial science, and queueing theory.
\end{abstract}

\section{Introduction}

In this paper, we develop a strongly efficient importance sampling scheme for
computing rare-event probabilities involving path functionals of heavy-tailed
random walks and compound Poisson processes in a general large deviations
regime. Heavy-tailed distributions play an important role in many man-made
stochastic systems. For example, they accurately model inputs to computer and
communications networks (see e.g.\ \cite{fosskorshunovzachary2013}), and they
are an essential component of the description of many financial risk processes
(see e.g.\ \cite{embrechtsklueppelbergmikosch1997}).

We focus on stochastic processes with a regularly varying increment/jump-size
distribution. The estimators produced with our sampling scheme are
straightforward to implement and can be used to estimate the likelihood of a
wide range of rare events with bounded relative error. In particular, such a
single sampling scheme applies to a very general class of rare events whose
occurrence is caused by one or several components in the system which exhibit extreme behavior, while the rest of the
system is operating in \textquotedblleft normal\textquotedblright%
\ circumstances (therefore, beyond the usual principle of a single big jump).
In particular, our results apply to a large class of continuous functionals of
multiple random walks and compound Poisson processes.

Our estimators are based on importance sampling, a Monte Carlo technique which
consists in biasing the nominal distribution of the underlying process in
order to induce the rare event of interest. The estimator is obtained by
weighting each sample by the corresponding likelihood ratio in order to obtain
unbiased estimators. Our goal is to find biasing techniques leading to
estimators which have bounded coefficient of variation uniformly as the
probability of the event of interest tends to zero in a suitable large
deviations regime. A brief review of importance sampling and the notion of
strong efficiency will be given later in this paper; for a more in-depth
discussion, see \cite{asmussenglynn2007}.

The construction of our sampling scheme is driven by recently developed sample
path large deviations results in \cite{rheeblanchetzwart2016} for regularly varying L\'{e}vy processes and
random walks. Specifically, let
$X(t)$, $t\geq0$ be a one-dimensional compensated compound Poisson process
with unit arrival rate and a positive jump distribution $W$ that is regularly
varying at infinity. 
Define $\bar{X}_{n}=\{\bar{X}_{n}(t)\}_{t\in\lbrack0,1]}%
$, with $\bar{X}_{n}(t)=X(nt)/n$. For a measurable set $A\subseteq\mathbb{D}$
satisfying a specific topological property, the large deviations results
derived in \cite{rheeblanchetzwart2016} establish that
\[
\mathbb{P}(\bar{X}_{n}\in A)=\Theta\left(  \big(n\mathbb{P}(W\geq
n)\big)^{l^{\ast}}\right),
\]
where precise details can be found in Section \ref{SecPR} below. In practice,
exact estimates are often demanded. Hence, we design a sampling scheme for
rare events that take the form $\mathbb{P}(\bar{X}_{n}\in A)$. 
We illustrate our approach with several
applications that arise in the context of mathematical finance, actuarial
science, and queueing theory.

In order to contextualize our contribution, let us provide a review of the
theory and methods which are standard in rare event simulation settings
similar to those studied in this paper.

In the context of stochastic processes with light-tailed characteristics, such
as random walks with increments possessing a finite moment generating function
in a neighborhood of the origin, large deviations theory can be used to bias
the process of interest in order to induce the occurrence of the rare event in
question. 
In fact, it is well known that a conventional type of proof of the asymptotic lower bound in large deviations analysis one can extract an exponential change of measure which can sometimes be proved to be efficient (for
counterexamples see e.g.\ \cite{glassermankou1995} and
\cite{glassermanwang1997}). By connecting the
design of efficient importance sampling estimators with a game theoretic
formulation, \cite{dupuissezerwang2007}, \cite{dupuiswang2004} and
\cite{dupuiswang2005} provide the foundations for the use of large deviations
theory in the construction and analysis of provably efficient rare event
simulation estimators. Moreover, a weakly efficient \textquotedblleft
universal\textquotedblright\ sampler has been proposed by
\cite{dupuiswang2009} for a general class of hitting sets in arbitrary Jackson
network topologies.

The setting of stochastic processes with heavy-tailed increments brings up
additional challenges compared to its light-tailed counterpart discussed in
the previous paragraph. These challenges were exposed in
\cite{asmussenschmidlischmidt1999}. First of all, typically, the asymptotic
conditional distribution of any particular increment given the rare event of
interest converges to the underlying nominal distribution. Intuitively, if a
rare event is caused by a large jump that may occur in a single
\textquotedblleft unlucky\textquotedblright\ increment out of many possible
alternatives, then the chance that any specific increment is precisely the
unlucky one is, naturally, small. So, any particular increment is likely to
behave \textquotedblleft normally\textquotedblright\ and therefore, in
contrast to the light-tailed setting, there is no direct way in which one
might attempt to bias a particular increment in order to stir the process
towards the rare event of interest.

Moreover, as pointed out in \cite{asmussenschmidlischmidt1999}, the asymptotic
description of the most likely way in which a rare event may occur, for
example due to the presence of a single large jump, does not lead to a valid
change of measure for importance sampling because it is possible that several
large jumps (or no large jump at all)\ might actually produce the event of
interest under the nominal dynamics of the system. In other words, the natural
biasing mechanism induced by directly approximating the zero-variance
importance sampling distribution in the heavy-tailed setting assigns zero
probability to events which are possible under the nominal dynamics leading to
an ill-defined likelihood ratio.

The use of state-dependent importance sampling provides a way to deal with
these difficulties. In \cite{blanchetglynn2008}, the authors explain how
approximating Doob's $h$-transform can lead to a feasible change of measure
which produces a strongly efficient importance sampling estimator in the
setting of first passage time probabilities for one dimensional random walks.
A Lyapunov technique was introduced for the analysis of state-dependent
importance sampling estimators. But the direct approximation of Doob's
$h$-transform might be difficult to implement in higher dimensions both
because of sampling implementation challenges and the evaluation of
normalizing constants.

In the setting of one-dimensional compound sums of independent and identically distributed (i.i.d.) regularly varying
random variables, \cite{dupuislederwang2007} produced a state-dependent change
of measure whose normalizing constant is straightforward to implement. Their idea can be
described as follows: each increment is sampled by either the original measure or---with small probability, which is a design parameter---a different measure, which is essentially the original measure conditional on exhibiting a large jump. 
The advantage of the mixture samplers is that sampling implementation challenges
and the evaluation of normalizing constants can often be addressed by choosing
a suitable set of parameters.

Under the setting where the time horizon is growing in large and moderate
deviation schemes, Blanchet and Liu show in \cite{blanchetliu2008} how to use
Lyapunov inequalities to address both the parameter selection while enforcing
a bounded relative error. A key step in the methodology is the
construction of a suitable Lyapunov function (for an illustration of the
technique in multidimensional settings, see \cite{blanchetliu2010}). Blanchet
and Liu suggest using the type of fluid analysis which is prevalent in the
large deviations literature of heavy-tailed stochastic processes (see
e.g.\ \cite{fosskorshunov2006} and \cite{fosskorshunov2012}). However, the
construction of the Lyapunov function and the verification of the Lyapunov
inequality becomes highly non-trivial in settings involving multiple jumps and
the presence of boundaries which are common in queueing systems, for an
example of the types of complications which arise in queueing settings, see
\cite{blanchetglynnliu2007}.

The idea of using mixtures, suggested in \cite{dupuislederwang2007}, is also
used here. But, while \cite{dupuislederwang2007} treats a particular
one-dimensional setting involving a rare event that is caused by a single big
jump during a bounded time horizon, our setting is more general. We allow for
a wide range of events, which might be caused by multiple jumps during a
growing time horizon in a large deviations scaling.

Recall that we are interested in estimating the probability $\mathbb{P}%
(\bar{X}_{n}\in A)$. The concept behind our sampling scheme can be described
as follows. Based on the large deviations results derived in
\cite{rheeblanchetzwart2016}, we construct first an auxiliary set $B^{\gamma}$
that is closely related to the optimization problem given by \eqref{SecMReq3}
below. 
Then, given a fixed mixture probability parameter $w\in(0,1)$, we generate the sample path of
$\bar{X}_{n}$ under the nominal measure.
And, with probability $1-w$ we generate the
sample path of $\bar{X}_{n}$ under the measure $\mathbb{Q}^{\gamma}%
(\,\cdot\,)\triangleq\mathbb{P}(\,\cdot\,|\bar{X}_{n}\in B^{\gamma})$. 
Finally, as a consequence of applying the importance sampling
technique, we scale our samplers with a suitable likelihood ratio given as in
\eqref{SecMReq1} below. It should be noted that the set $A$ can be as general
as in the setting of \cite{rheeblanchetzwart2016}. 
Therefore, our methodological contribution in this paper addresses precisely those types of
difficulties mentioned in the previous paragraphs, such as, multiple jumps, growing time horizon, avoiding the evaluation of normalizing constants, and by-passing the verification of Lyapunov inequalities. 
The advantages of our
sampling scheme are that the new estimators are strongly efficient and
straightforward to implement. Moreover, they are \textquotedblleft
universal\textquotedblright\ in the sense that a single importance sampling
scheme applies to a very general class of rare events involving multiple jumps
that arise in heavy-tailed systems. 
As a final remark, it should be
mentioned that constructing the auxiliary set $B^{\gamma}$ requires choosing a
set of suitable parameters $\gamma$ whose existence is guaranteed by the
topological property we impose on $A$. Hence, one of the main challenges is to
select the set of parameters specifically for each application.

Our mathematical contributions in this paper can be summarized as follows.

\begin{itemize}
\item We propose a simulation algorithm for estimating the rare-event
probability of $\bar{X}_{n}\in A$, together with a sampling scheme for
$\bar{X}_{n}\in\cdot$ given $\bar{X}_{n}\in B^{\gamma}$, which is based on a
rejection sampling with an unconditional acceptance probability bounded away
from zero as $n\to\infty$.

\item By showing the existence of the parameter $\gamma$, we prove the
strong efficiency of our sampling scheme under a very general setting (see
Assumption \ref{SecMRass2} below).

\item We showcase the versatility of the algorithm by illustrating the implementation of the proposed
sampling scheme to the rare-events that arise in finance, actuarial science, and queueing theory. 

\item Especially, in the application to queueing networks (see Section
\ref{SecAAQN} below), we show that the tail index of the rare-event
probability---which usually exhibit a complex boundary
behavior due to the nonlinear nature of the associated Skorokhod mapping---can be determined by solving knapsack problem with nonlinear
constraints.
\end{itemize}

The rest of the paper is organized as follows. Section \ref{SecPR} deals with basic
background and notations required to state our contributions.
Section \ref{SecMR} introduces our estimators and describes the main result.
Applications and numerical implementations are discussed in Section~\ref{SecAAAS}, Section~\ref{SecAABOP}, and Section~\ref{SecAAQN}. All the proofs of results presented in this paper are given
in Section~\ref{SecP}.

\section{Notations and preliminaries}\label{SecPR}
\subsection{Notations}
We start with a summary of notations that will be employed in this paper.
Let $\mathbb{Z}_+$ denote the set of non-negative integers, and let $\mathbb{R}_+$ denote the set of non-negative real numbers. 
Let $A^\circ$ and $A^-$ denote the interior and the closure of $A$, respectively.
Let $(\mathbb{D}_{[0,1],\mathbb{R}},d)$ be the metric space of real-valued RCLL functions on $[0,1]$, denoted by $\mathbb{D}=\mathbb{D}_{[0,1],\mathbb{R}}$, equipped with the Skorokhod $J_1$ metric on $\mathbb D$ that is defined by
\[
d(x,y)=\inf_{\lambda\in\Lambda} ||\lambda-id||_\infty\vee||x\circ\lambda-y||_\infty,\quad x,y\in\mathbb D,
\]
where $id$ denotes the identity mapping, $||\cdot||_\infty$ denotes the uniform metric, i.e., $\|x\|_\infty \triangleq \sup_{t\in[0,1]} |x(t)|$, and $\Lambda$ denotes the set of all strictly increasing, continuous bijections from $[0,1]$ to itself. 
Let $\mathbb{D}^k$ denote the $k$-fold product space of $\mathbb D$.
Let $\mathbb{D}_\uparrow^k$ denote the subset of functions in $\mathbb D^k$ that are non-negative and nondecreasing in each coordinate.
When it comes to the tail indices of a regularly varying distribution, we use $\beta$ (or $\beta_i$ in the multidimensional case) for the right tail and $\alpha$ for the left tail.
Let $\mathbb{D}_l$ denote the subspace of $\mathbb D$ consisting of non-decreasing step functions vanishing at time zero with $l$ jumps, and let $\mathbb{D}_{< l^*}$ denote the subspace of $\mathbb D$ consisting of non-decreasing step functions vanishing at $0$ with at most $l-1$ jumps, i.e.\ $\mathbb{D}_{< l^*}=\bigcup_{l\leq l^*-1}\mathbb{D}_l$.
Define
\[
\mathbb{D}_{<(l_1^*,\ldots,l_d^*)} \triangleq \bigcup_{(l_1,\ldots,l_d)\in I_{<(l_1^*,\ldots,l_d^*)}}\prod_{i=1}^d\mathbb{D}_{l_i},
\]
where
\[
I_{<(l_1^*,\ldots,l_d^*)}\triangleq\left\{ (l_1,\ldots,l_d)\in\mathbb{Z}_+^d\setminus\{(l_1^*,\ldots,l_d^*)\} \,\middle|\, \mathcal{I}(l_1,\ldots,l_d)\leq\mathcal{I}(l_1^*,\ldots,l_d^*) \right\},
\]
and $\mathcal{I}(l_1,\ldots,l_d)\triangleq (\beta_1-1)l_1+\ldots+(\beta_d-1)l_d$.
Define a partial order $\prec$ on $\mathbb Z_+^d$ such that $(l_1,\ldots,l_d)\prec(m_1,\ldots,m_d)$ if and only if $\mathbb{C}_{(l_1,\ldots,l_d)}\varsubsetneq\mathbb{C}_{(m_1,\ldots,m_d)}$, where $\mathbb{C}_{(l_1,\ldots,l_d)}\triangleq\bigcup_{i=1}^d\, \mathbb{D}^{i-1}\times\mathbb{D}_{< l_i}\times\mathbb{D}^{d-i}$.
Define
\begin{align*}
J_{(j_1,\ldots,j_d)} \triangleq \left\{ (l_1,\ldots,l_d)\in\mathbb Z_+^d\setminus I_{<(j_1,\ldots,j_d)} \,\middle|\, (m_1,\ldots,m_d)\prec(l_1,\ldots,l_d) \text{\ implies\ } (m_1,\ldots,m_d)\in I_{<(j_1,\ldots,j_d)} \right\}.
\end{align*}
Let $\mathbb D_{l_-;l_+}$ denote the subspace of the Skorokhod space consisting of step functions vanishing at the origin with exactly $l_-$ downward jumps and $l_+$ upward jumps.
Define
\[
\mathbb D_{<l_-^*;l_+^*} \triangleq \bigcup_{(l_-,l_+)\in I_{<l_-^*;l_+^*}} \mathbb D_{l_-;l_+}
\]
where $I_{<l_-^*;l_+^*} \triangleq \left\{ (l_-,l_+)\in\mathbb Z_+^2\setminus\{(l_-^*,l_+^*)\} \,\middle|\, (\alpha-1)l_- + (\beta-1)l_+\leq(\alpha-1)l_-^* + (\beta-1)l_+^* \right\}$.

Given non-negative sequences of real numbers $x_n$ and $y_n$, we write $x_n=\mathcal{O}(y_n)$, $x_n=\omicron(y_n)$ and $x_n=\Theta(y_n)$, if $\limsup_{n\to\infty} x_n/y_n < \infty$, $\lim_{n\to\infty} x_n/y_n = 0$ and $0 < \liminf_{n\to\infty} x_n/y_n\leq\limsup_{n\to\infty} x_n/y_n < \infty$, respectively.
Given two $\mathbb R$-valued functions $f$ and $g$, we write $f\propto g$, if there exists $c\in\mathbb R$ such that $f=cg$.
For $x=(x_1,\ldots,x_k)$, $y=(y_1,\ldots,y_k)\in\mathbb R^k$, we write $x\leq y$, if $x_i\leq y_i$, for all $i\in\{1,\ldots,k\}$.
Let the cardinality of $S$ be denoted by $| S |$ or $\#S$. Finally, let $\mathcal{C}(S,k)$ and $\mathcal{P}(S,k)$ denote the set of all $k$-combinations and $k$-permutations of a set $S$, respectively. Note that $|C(S,k)| = {|S| \choose k}$ and $|D(S,k)|=|C(S,k)|*k!$ 

To describe the efficiency of a rare-event simulation algorithm, we adopt a widely applied criterion, which requires that the relative mean squared error of the associated estimator is appropriately controlled. 
To be more precise, suppose that we are interested in a sequence of rare events $A_n$, which becomes more and more rare as $n\to\infty$. 
For each $n\in\mathbb{Z}_+$, let $L_n$ be an unbiased estimator of the rare-event probability $\eta_n=\P(A_n)$. 
$L_n$ is said to be strongly efficient if $\mathbb{E}L_n^2=\mathcal{O}(\eta_n^2)$. 
In particular, strong efficiency implies that the number of simulation runs required to estimate the target probability to a given relative accuracy is bounded with respect to (w.r.t.) $n$.

\subsection{Preliminaries}
As we will see, the simulation algorithm that we propose in this paper is constructed based on the asymptotic behavior of rare-event probabilities, therefore we review
some well known large deviations results for scaled L\'evy processes with heavy-tailed L\'evy measures, introduced in \cite{rheeblanchetzwart2016}.
To begin with, let $X$ be a L\'evy process with L\'evy measure $\nu$, where $\nu$ is spectrally positive and regularly varying (at infinity) with index $-\beta<-1$. 
Let $\bar{X}_n\triangleq\left\{ X(nt)/n \right\}_{t\in[0,1]}$ denote the associated scaled process. 
Let $\nu_\beta^l$ denote the restriction of $l$-fold product measure of $\nu_\beta$ to $\{x \in \mathbb R_+^l: x_1 \geq x_2 \geq \ldots \geq x_l\}$, where $\nu_\beta(x,\infty)\triangleq x^{-\beta}$. 
For $l\geq1$, define a (Borel) measure
\[
C_l(\,\cdot\,)\triangleq \E \left[ \nu_\beta^l \Big\{ y\in(0,\infty)^l \,\Big|\, \sum_{i=1}^l y_i \1_{[U_i,1]}\in\,\cdot\, \Big\} \right]
\]
where $U_i$, $i\geq1$ are i.i.d.\ uniformly distributed on $[0,1]$.
Note that $C_l$ is concentrated on $\mathbb D_l$, i.e., $C_l(\mathbb D_l)=1$.
Moreover, we make the convention that $C_0$ is the Dirac measure concentrated on the zero function. 
The following result is useful in designing an efficient algorithm for rare events involving one-dimensional scaled processes.
Throughout the rest of this paper, all measurable sets are understood to be Borel measurable.

\begin{result}[Theorem 3.1 of \cite{rheeblanchetzwart2016}]\label{SecPRresult2}
Suppose that $A$ is a measurable set. If $A$ is bounded away from $\mathbb D_{< l^*}$, where $l^* \triangleq \min \left\{ l\in\mathbb{Z}_+ \,\middle|\, \mathbb{D}_l \cap A \neq \emptyset \right\}<\infty$,
then we have that
\[
C_{l^*}(A^\circ)
\leq
\liminf_{n\rightarrow\infty} \frac{\P(\bar X_n \in A) }{(n \nu[n,\infty))^{l^*}}
\leq
\limsup_{n\rightarrow\infty} \frac{\P(\bar X_n \in A)}{(n \nu[n,\infty))^{l^*}}
\leq
C_{l^*}(A^-).
\]
\end{result}

As one can see in Section \ref{SecAABOP} and Section \ref{SecAAQN} below, plenty of applications can be interpreted as sample-path rare events in a multidimensional setting. 
Therefore, it is particularly interesting to consider large deviations results for multidimensional processes. 
Let $X^{(1)},\ldots, X^{(d)}$ be independent centered one-dimensional L\'evy processes with spectrally positive L\'evy measures $\nu_1(\cdot), \ldots,\nu_d(\cdot)$, respectively, where each $\nu_i$ is regularly varying with index $-\beta_i<-1$ at infinity. 
Moreover, for the finite product of metric spaces, we use the maximum metric; i.e., we use $d_{\mathbb S_1\times\cdots\times\mathbb S_d}((x_1,\ldots,x_d), (y_1,\ldots,y_d)) \triangleq \max_{i=1,\ldots,d}d_{\mathbb S_i}(x_i,y_i) $ for the product $\mathbb S_1\times\cdots\times\mathbb S_d$ of metric spaces $(\mathbb S_i,d_{\mathbb S_i})$. Finally, for $(l_1,\ldots,l_d)\in\mathbb Z_+^d$, we define $C_{l_1} \times\cdots\times C_{l_d}(\cdot)$ (which is concentrated on $\prod_{i=1}^d \mathbb D_{l_i}$) as the product measure of
\[
C_{l_i}(\,\cdot\,)\triangleq \E \left[ \nu_{\beta_i}^{l_i} \Big\{ y\in(0,\infty)^{l_i} \,\Big|\, \sum_{j=1}^{l_i} y_j \1_{[U_j,1]}\in\,\cdot\, \Big\} \right].
\]
Result~\ref{SecPRresult3} states a large deviations result for $d$ dimensional process $\bar X_n(t) \triangleq (X^{(1)}(nt)/n, \ldots, X^{(d)}(nt)/n)$ for $t\in[0,1]$.

\begin{result}[Theorem 3.6 of \cite{rheeblanchetzwart2016}]\label{SecPRresult3}
Suppose that $A$ is measurable. If $A$ is bounded away from $\mathbb{D}_{<(l_1^*,\ldots,l_d^*)}$, where
\begin{equation}\label{SecMReq3}
(l_1^*,\ldots,l_d^*) = \argmin_{\substack{(l_1,\ldots,l_d)\in\mathbb{Z}_+^d\\\prod_{i=1}^d\mathbb{D}_{l_i}\cap A\neq\emptyset}} \mathcal{I}(l_1,\ldots,l_d),
\end{equation}
then we have that
\begin{align*}
C_{l_1^*} \times\cdots\times C_{l_d^*}(A^\circ)
&\leq
\liminf_{n\rightarrow\infty} \frac{\P(\bar X_n \in A) }{\prod_{i=1}^d\big(n\nu_i[n,\infty)\big)^{l_i^*}}\\
&\leq\limsup_{n\rightarrow\infty} \frac{\P(\bar X_n \in A)}{\prod_{i=1}^d\big(n\nu_i[n,\infty)\big)^{l_i^*}}
\leq
C_{l_1^*} \times\cdots\times C_{l_d^*}(A^-).
\end{align*}
\end{result}
Note that the assumption that $A$ is bounded away from $\mathbb{D}_{<(l_1^*,\ldots,l_d^*)}$ guarantees the uniqueness of $(l_1^*,\ldots, l_d^*)$.
Finally, 
we conclude this section with an extension of Result \ref{SecPRresult3}, which will be useful in
constructing an efficient simulation algorithm for heavy-tailed random walks.
Let $S_k$, $k\geq 0$, be a random walk, set $\bar S_n(t) = S_{\lfloor nt \rfloor}/n$, $t\geq 0$, and define $\bar S_n = \{\bar S_n(t)$, $t\in [0,1]\}$. 
Let $\nu_\beta^l$ be as defined above.
Similarly, let $\nu_\alpha^m$ denote the restriction of $m$-fold product measure of $\nu_\alpha$ to $\{x \in \mathbb R_+^m: x_1 \geq x_2 \geq \ldots \geq x_m\}$, where $\nu_\alpha(x,\infty)\triangleq x^{-\alpha}$.
Let $C_{0,0}(\cdot) \triangleq \delta_{\mathbf 0}(\cdot)$ be the Dirac measure concentrated on the zero function. For each $(l_-,l_+)\in \mathbb Z_+^2\setminus \{(0,0)\}$, define a measure (which is concentrated on $\mathbb D_{l_-;l_+}$)
\[
C_{l_-;l_+}(\cdot) \triangleq \E \Big[\nu_\alpha ^{l_-}\times\nu_\beta^{l_+} \{(x,y)\in (0,\infty)^{l_-}\times(0,\infty)^{l_+}:\sum_{i=1}^{l_-} x_i 1_{[U_i,1]} - \sum_{i=1}^{l_+} y_i1_{[V_i,1]}\in \cdot\}\Big],
\]
where $U_i$'s and $V_i$'s are i.i.d.\ uniform on $[0,1]$. 
\begin{result}\label{SecPRresult4}
Suppose that $\P(S_1 \leq -x)$ is regularly varying with index $-\alpha$ and $\P(S_1 \geq x)$ is regularly varying with index $-\beta$. Let $A$ be a measurable set bounded away from $\mathbb D_{< l_-^*;l_+^*}$,
where
\begin{equation}\label{SecMReq40}
(l_-^*,l_+^*) = \argmin_{\substack{(l_-,l_+)\in\mathbb{Z}_+^2\\ \mathbb{D}_{l_-;l_+}\cap A\neq\emptyset}} (\alpha-1)l_- + (\beta-1)l_+.
\end{equation}
Then
\begin{align*}
C_{l_-^*;l_+^*}(A^\circ)
\leq
&\liminf_{n\rightarrow\infty} \frac{\P(\bar S_n \in A) }{(n \P(S_1\leq -n))^{l_-^*} (n \P(S_1\geq n)))^{l_+^*}}  \\
\leq
&\limsup_{n\rightarrow\infty} \frac{ \P(\bar S_n \in A)}{(n \P(S_1\leq -n))^{l_-^*} (n \P(S_1\geq n)))^{l_+^*}}
\leq
C_{l_-^*;l_+^*}(A^-).
\end{align*}
\end{result}

\section{Main results}\label{SecMR}
In this section we present our main results. Although the large deviations results reviewed in Section \ref{SecPR} are stated for L\'evy processes, we focus on compensated compound Poisson process for simulation purposes. 
Let $X$ denote a $d$-dimensional compensated compound Poisson process, and recall that $\bar{X}_n$ is the scaled process with $\bar{X}_n(t)=X(nt)/n$, $t\in[0,1]$. 
For a measurable set $A\in\mathbb D^d$, we are interested in estimating the probability of the event $A_n \triangleq \{\bar{X}_n\in A\}$, when $n$ is large. 
Note that, in view of the law of large numbers, one can expect that $\P(\bar X_n \in A) \to 0$ for $A$'s that are bounded away from the zero function, and hence, $A_n$'s are rare events for large $n$'s. 
In Section \ref{SecASC}, we first illustrate the idea of our algorithm in the special case for $d=1$, where the notations are simpler. In Section \ref{SecEG} we extend this result to general $d$.

\subsection{The one-dimensional case}\label{SecASC}
Let $\{X(t)\}_{t\geq0}$ be a one-dimensional compensated compound Poisson process with jump sizes $\{W(k)\}_{k\geq1}$. That is,
\[X(t)=\sum_{k=1}^{N(t)} W(k)-\lambda t\E W(1)\]
where $\{N(t)\}_{t\geq0}$ is a Poisson process with arrival rate $\lambda$, and let $\bar{X}_n\triangleq\left\{ X(nt)/n \right\}_{t\in[0,1]}$ denote the associated scaled process. Moreover, let $\P(X(1)>x)$ be regularly varying of index $-\beta<-1$. The following assumption is essential for analyzing the asymptotic behavior of the rare-event probability, and hence, deriving the strong efficiency of our estimator. 

\begin{assumption}\label{SecMRass1}
Let $A$ be a measurable set in $\mathbb D$. We assume that $A$ is bounded away from $\mathbb{D}_{< l^*}$, where $l^* = \min \left\{ l\in\mathbb{Z}_+ \,\middle|\, \mathbb{D}_l \cap A \neq \emptyset \right\}$ denotes the minimal number of upward jumps of a step function in $A$. Moreover, assume that $C_{l^*}(A^\circ)>0$.
\end{assumption}

\begin{rmk}
As one can see in Section~\ref{SecAAAS}, \ref{SecAABOP}, and \ref{SecAAQN}, one of the typical settings that arises in applications is that the set $A$ can be written as a finite combination of unions and intersections of $F_1^{-1}(A_1),\ldots,F_m^{-1}(A_m)$, where each $F_i\colon\mathbb{D}\to\mathbb{S}_i$ is a continuous function, and all sets $A_i$ are subsets of a general topological space $\mathbb{S}_i$. If we denote this operation of taking unions and intersections by $\Psi$ (i.e., $A=\Psi ( F_1^{-1}(A_1),\ldots,F_m^{-1}(A_m) )$), then it holds that
\[\Psi\left( F_1^{-1}(A_1^\circ),\ldots,F_m^{-1}(A_m^\circ) \right) \subseteq A^\circ \subseteq A \subseteq A^- \subseteq \Psi\left( F_1^{-1}(A_1^-),\ldots,F_m^{-1}(A_m^-) \right).\]
Hence, $C_{l^*}(A^\circ)>0$ holds if $\hat{T}_{l^*}^{-1}\left( \Psi\left( F_1^{-1}(A_1^\circ),\ldots,F_m^{-1}(A_m^\circ) \right) \right)$ has positive Lebesque measure, where $\hat{T}_{j}\colon \hat{S}_j\to\mathbb{D}_j$ is defined by $\hat{T}_{j}(x,u)\triangleq \sum_{i=1}^j x_i\1_{[u_i,1]}$ for $j\in\mathbb{Z}_+$, and
\[\hat{S}_j\triangleq\left\{ (x,u)\in\mathbb{R}_+^j\times[0,1]^j \,\middle|\, x_1\geq \cdots \geq x_j, 0,1,u_1,\ldots,u_j\text{\ are distinct}\right\}.\]
Analogously, one can derive a sufficient condition for $C_{l_1^*}\times\cdots\times C_{l_d^*}(A^\circ)>0$ (see Assumption \ref{SecMRass2} below). More details about this discussion can be found in Section 3.1 of \cite{rheeblanchetzwart2016}.
\end{rmk}

We design a rare-event simulation algorithm that estimates the probability of $A_n\triangleq\left\{\bar{X}_n\in A\right\}$ efficiently, based on an importance sampling strategy.
To construct an importance distribution, we introduce a constant $\gamma>0$ and define $B_n^\gamma\triangleq \left\{ \bar{X}_n\in B^\gamma \right\}$, where $B^\gamma$ is given by
\[B^\gamma \triangleq \left\{ \xi \,\middle|\, \#\big\{ t \,\big|\, \xi(t)-\xi(t^-)>\gamma \big\}\geq l^* \right\}. \]
In the construction of our rare-event simulation algorithm, we will take advantage of the fact that one can always choose $\gamma$ so that $B_n^\gamma$ is sufficiently ``close'' to $A_n$.
The specific choice of $\gamma$ will be further discussed later in Section~\ref{SecAAAS}, \ref{SecAABOP}, and \ref{SecAAQN} for concrete examples.
Let $\Q_\gamma(\,\cdot\,) \triangleq \P(\,\cdot\,| B_n^\gamma)$ denote the conditional distribution given $\bar{X}_n\in B^\gamma$. One should notice that $d\Q_\gamma/d\P = \P(B_n^\gamma)^{-1}\1_{B_n^\gamma}$.
Moreover, by the Fubini-Tonelli's theorem, a closed-form expression for $\P(B_n^\gamma)$ is given by
\begin{align}
\nonumber\P\left(B_n^\gamma\right)&=e^{-\lambda n}\sum_{i=l^*}^\infty \frac{(\lambda n)^i}{i!} \sum_{j=l^*}^i {i \choose j} \P(W(1)>n\gamma)^j\left(1-\P(W(1)> n\gamma)\right)^{i-j}\\
\nonumber&=e^{-\lambda n}\sum_{j=l^*}^\infty \frac{(\lambda n)^j}{j!}\P(W(1)>n\gamma)^j \sum_{i=j}^\infty \frac{(\lambda n)^{i-j}}{(i-j)!} \left(1-\P(W(1)>n\gamma)\right)^{i-j}\\
&=1-\exp\bigg\{- \lambda n\P(W(1)>n\gamma)\bigg\} \sum_{j=0}^{l^*-1} \frac{(\lambda n)^j}{j!}\P(W(1)> n\gamma)^j.\label{SecMReq4}
\end{align}
From \eqref{SecMReq4} one should recognize that $B_n^\gamma$ can be interpreted as the event of a Poisson distributed random variable with rate $\lambda n\P(W(1)>\gamma n)$ crossing the level $l^*$. Now, let $w\in(0,1)$ be arbitrary but fixed. We propose an importance distribution $\Q_{\gamma,w}$ that is absolutely continuous  w.r.t.\ $\P$ and is given by
\begin{equation}\label{SecMReq10}
\Q_{\gamma,w}(\,\cdot\,)\triangleq w \P(\,\cdot\,) + (1-w)\Q_\gamma(\,\cdot\,).
\end{equation}

We give here an algorithm for generating the sample path of $\bar X_n$ under the probability measure $\Q_\gamma(\,\cdot\,)$. Since $\{\bar X_n\in B^\gamma\}\subseteq\{N(n)\geq l^*\}$, we observe that
\begin{align*}
\Q_\gamma(\bar X_n\in\,\cdot\,)&=\frac{1}{\P(B_n^\gamma)}\P(\bar X_n\in\,\cdot\,,B_n^\gamma)=\sum_{m=l^*}^{\infty} h_m\,\P(\bar X_n\in\,\cdot\,|\,B_n^\gamma,N(n)=m),
\end{align*}
where $h_m=h_m(n)\triangleq \P(B_n^\gamma,N(n)=m)/\P(B_n^\gamma)$ satisfies $\sum_{m\geq l^*} h_m=1$.
Note that $h_m$ can be computed, since
\begin{align*}
\P(B_n^\gamma,N(n)=m)&=\P(B_n^\gamma\,|\,N(n)=m)\P(N(n)=m)\\
&=e^{-\lambda n}\frac{(\lambda n)^m}{m!}\left( \sum_{i=l^*}^m {m \choose i} \P(W(1)>n\gamma)^i\left(1-\P(W(1)>n\gamma)\right)^{m-i} \right).
\end{align*}
Hence, it remains to discuss sampling from $\P(\bar X_n\in\,\cdot\,|\,B_n^\gamma,N(n)=m)$.
It turns out that we can proceed a rejection sampling, where drawing from the proposal distribution can be achieved as follows.
\begin{enumerate}
\item Sample $\{b_k\}_{k\leq l^*}$ uniformly from $\mathcal{C}\left( \left\{1,\ldots,m\right\},l^* \right)$;
\item sample each $W(b_k)$, $k\leq l^*$, conditional on $W(1) > n \gamma$;
\item sample $W(m')$, $m'\leq m$, $m'\notin\{b_k\}_{k\leq l^*}$, under the nominal measure.
\end{enumerate}
Note that the target density $f_{\text{target};m}$, defined by
\[
f_{\text{target};m}\left( w_1,\ldots, w_m \right)dw_1\cdots dw_m \triangleq \P\left( W(1)\in w_1+dw_1,\ldots,W(m)\in w_m+dw_m \,\middle|\, B_n^\gamma,N(n) = m \right),
\]
can be bounded by $M_m f_{\text{proposal};m}\left( w_1,\ldots, w_m \right)$, where
\begin{align*}
f_{\text{target};m}\left( w_1,\ldots, w_m \right)&\propto\P(B_n^\gamma|N(n)=m)^{-1}\prod_{j=1}^{m} \frac{d}{d w_j}\P\left( W(j)\leq w_j \right)\1_{B_n^\gamma}(w_1,\ldots,w_m),
\end{align*}
\begin{align*}
f_{\text{proposal};m}\left( w_1,\ldots, w_m \right)&=\binom{m}{l^*}^{-1}\P\left( W(1) > n \gamma \right)^{-l^*} \prod_{j=1}^{m} \frac{d}{d w_j}\P\left( W(j)\leq w_j \right) \sum_{\substack{(b_1,\ldots,b_{l^*})\in\\\mathcal C(\{1,\ldots,m\},l^*)}} \1_{\{W(b_k)>n\gamma,\forall k\leq l^*\}},
\end{align*}
and hence,
\[
M_m=M_m(n)\triangleq\frac{\binom{m}{l^*}\P(W(1)> n\gamma)^{l^*}}{\P(B_n^\gamma|N(n)=m)}.
\]
Now, it is natural to accept $(W(1),\ldots,W(m))$ with probability
\[
a( W(1),\ldots,W(m) ) = \binom{\#\left\{i\in\{1,\ldots,m\} \,\middle|\, W(i) > n\gamma \right\}}{l^*}^{-1}.
\]
Finally, we are able to formulate the pseudocode for generating the sample path of $\bar X_n$ under $\Q_\gamma$ in Algorithm \ref{SecMRalg20}. Moreover, we show in Proposition \ref{SecMRprop20} that the expected running time of Algorithm \ref{SecMRalg20} is uniformly bounded from above w.r.t.\ $n$.

\begin{algorithm}
\caption{Generating the sample path of $\bar X_n$ under $\Q_\gamma$}
\label{SecMRalg20}
\begin{algorithmic}
\renewcommand{\algorithmicrequire}{\textbf{Input:}}

\REQUIRE $\gamma$

\STATE sample $m \sim h_m$ \COMMENT{$m=m'$ with probability $h_{m'}=\P(N(n)=m' \,|\, B_n^\gamma)$}

\STATE $R\leftarrow \TRUE$

\WHILE{$R=\TRUE$}

\STATE sample $\{b_k\}_{k\leq l^*} \sim \text{unif}\left( \mathcal{C}\left( \left\{1,\ldots,m\right\},k \right) \right)$ \COMMENT{uniform distribution on $\mathcal{C}\left( \left\{1,\ldots,m\right\},k \right)$}

\FOR{$i=1$ to $l^*$}
    \IF{$i\in \{b_k\}_{k\leq l^*}$}
    \STATE sample $W(i)\sim W(1)\,\big|\,W(1) > n \gamma$
    \ELSE
    \STATE sample $W(i)\sim W(1)$
    \ENDIF
\ENDFOR

\STATE $c\leftarrow \#\left\{j\in\{1,\ldots,m\} \,\middle|\, W(j) > n \gamma \right\}$
\STATE $a\leftarrow \binom{c}{l^*}^{-1}$

  \STATE sample $u\sim\text{uniform}[0,1]$
    \IF{$u<a$}
    \STATE $R\leftarrow \FALSE$
    \ELSE
    \STATE $R\leftarrow \TRUE$
    \ENDIF

\ENDWHILE

\RETURN $\bar{X}_n$

\end{algorithmic}
\end{algorithm}

\begin{prop}\label{SecMRprop20}
Let $T_{\text{alg\ref{SecMRalg20}}}(n)$ denote the expected running time of Algorithm \ref{SecMRalg20}. Under the assumption that $W(1)$ is regularly varying of index $-\beta<-1$, we have that $T_{\text{alg\ref{SecMRalg20}}}(n)=\sum_{m\geq l^*} h_m(n) M_m(n)$ is uniformly bounded from above w.r.t.\ $n$, i.e.\ $\max_{n\geq0}T_{\text{alg\ref{SecMRalg20}}}(n)<\infty$.
\end{prop}

\begin{proof}
See Section \ref{SecP}.
\end{proof}

In view of the observations we made so far, we propose an estimator $Z_n$ for $\P(A_n)$ that is given by
\begin{equation}\label{SecMReq1}
Z_n = \1_{A_n}\frac{d\P}{d\Q_{\gamma,w}} = \frac{\1_{A_n}}{w+\frac{1-w}{\P(B_n^\gamma)}\1_{B_n^\gamma}}.
\end{equation}
Intuitively, an importance sampling technique is used to get more samples from the interesting region, by sampling from a distribution that overweights the important region. Based on this, the choice of $B_n^\gamma$ can be ``justified'', since $B_n^\gamma$ is mimicking the asymptotic behavior of the probability of interest. However, as one can see in the proof of strong efficiency (see Theorem \ref{SecMRthm1} below), we should analyze the second moment of our estimator to avoid ``backfire'', yielding an estimator with larger or even infinite variance. It turns out that this intuition can be made rigorous by applying Result \ref{SecPRresult2}. We end this section with a theorem regarding to the strong efficiency of our estimator.

\begin{thm}\label{SecMRthm2}
Under Assumption \ref{SecMRass1}, there exists a $\gamma>0$ such that the estimator constructed in \eqref{SecMReq1} is strongly efficient for estimating $\P(A_n)$.
\end{thm}

\begin{proof}
Analogous to the proof of the more general Theorem \ref{SecMRthm1} presented below, where the existence of $\gamma$ is also discussed.
\end{proof}

\subsection{Extension to general $d$}\label{SecEG}
In this section we extend the results in Section \ref{SecASC} to the case with arbitrary $d$. 
To be precise, let $X\triangleq \left( X^{(1)},\ldots,X^{(d)} \right)$ be a superposition of $d$ independent compensated compound Poisson processes with upward jumps, where $\{N^{(i)}(t)\}$ is a Poisson process with arrival rate $\lambda_i$, and
\[X^{(i)}(t)=\sum_{k=1}^{N^{(i)}(t)} W^{(i)}(k)-\lambda_i t \E W^{(i)}(1).\]
Moreover, let $\P(X^{(i)}(1)>x)$ be regularly varying of index $-\beta_i<-1$ at infinity. 
Finally, let $\bar{X}_n$ denote the corresponding scaled process. As we can see in Result \ref{SecPRresult3}, the large deviations results for $\P \left( \bar{X}_n \in A \right)$ depend heavily on the value of $\mathcal I (l_1^*,\ldots,l_d^*)$, where $(l_1^*,\ldots,l_d^*)$ is as defined in \eqref{SecMReq3}. However, for $c\in\mathbb R$, the grid $(l_1,\ldots,l_d)\in\mathbb Z_+^d$ satisfying $\mathcal I (l_1,\ldots,l_d)=c$ is not unique in general. Therefore assuming $A$ being bounded away from $\prod_{i=1}^d \mathbb D_{<l_i}$ is not sufficient for our purposes. The following assumption, which is slightly different from the one we made in Section \ref{SecASC}, corresponds to the extension of Result \ref{SecPRresult2} to Result \ref{SecPRresult3}.

\begin{assumption}\label{SecMRass2}
Let $A$ be a measurable set. Assume that $A$ is bounded away from $\mathbb{D}_{<(l_1^*,\ldots,l_d^*)}$, where $(l_1^*,\ldots,l_d^*)$
is the unique solution of the minimization problem given by \eqref{SecMReq3}. Moreover, assume that $C_{l_1^*}\times\cdots\times C_{l_d^*}(A^\circ)>0$.
\end{assumption}

If the solution to \eqref{SecMReq3} is not unique, we may partition $A$. 
As in Section \ref{SecASC}, we focus now on constructing the auxiliary set $B^\gamma$ for the importance distribution. 
Define $A_n\triangleq\left\{\bar{X}_n\in A\right\}$ and $B_n^\gamma\triangleq\{\bar{X}_n\in B^\gamma\}$. 
As one can see in the proof of Theorem \ref{SecMRthm1}, controlling the probability of $A_n\cap (B_n^\gamma)^c$ should be taken into account in choosing the auxiliary set $B^\gamma$. 
In the one-dimensional case, letting $B^\gamma$ mimic the optimal path leading to the rare event makes us capable of controlling the relative error of our estimator. 
However, the same strategy will fail in the multidimensional case, since the rare event can be reached through other feasible but not necessarily optimal paths. 
Thus, we require a more complicated construction of $B^\gamma$.

\begin{defn}\label{SecMRdef1}
Let $A$ be a measurable set in $\mathbb{D}^d$, and let $(l_1^*,\ldots,l_d^*)$ denote the unique solution to \eqref{SecMReq3}.
Let $\gamma\in\mathbb R^d$ with $\gamma_i>0$ for all $i\in\{1,\ldots,d\}$, and define
\begin{equation}\label{SecMReq2}
B^\gamma\triangleq\bigcup_{(l_1,\ldots,l_d)\in J_{(l_1^*,\ldots,l_d^*)}} B^{ \gamma;l },
\end{equation}
where $B^{\gamma;l}$ is the set of c\`adl\`ag functions on $\mathbb R^d$ that have greater or equal to than $l_i$ number of jumps with size larger than $\gamma_i$ in its $i$-th coordinate, i.e.,
\[
B^{ \gamma;l } \triangleq \left\{ \left( \xi^{(1)},\ldots,\xi^{(d)} \right)\in\mathbb D^d \,\middle|\, \#\big\{ t \,\big|\, \xi^{(i)}(t)-\xi^{(i)}(t^-)>\gamma_i \big\}\geq l_i,\,\forall i\in\{1,\ldots,d\} \right\}.
\]
\end{defn}

\begin{rmk}
Note that the cardinality of $J_{(l_1^*,\ldots,l_d^*)}$ is finite.
To design a strongly efficient simulation algorithm for estimating $\P(A_n)$, we will take advantages of an important property of $J_{(l_1^*,\ldots,l_d^*)}$.
That is, for all $\xi\in A$ with $A$ being bounded away from $\mathbb D_{<(l_1^*,\ldots,l_d^*)}$, there exists an index $(l_1,\ldots,l_d)\in J_{(l_1^*,\ldots,l_d^*)}$, such that the path of $\xi$ in its $i$-th coordinate is bounded away from $\mathbb D_{<l_i^*}$, for every $i\in\{1,\ldots,d\}$.
An illustration of $J_{(l_1^*,\ldots,l_d^*)}$ can be found in Figure \ref{SecMRfig1}.
\end{rmk}

\begin{figure}[]
\centering

\begin{tikzpicture}

\tikzstyle{axes}=[]
\tikzstyle{important line}=[very thick]
        
\draw[style=help lines, step=0.5cm] (-0.1,-0.1) grid (4.1,2.6);
        
\begin{scope}[style=axes]
\draw[->] (-0.4,0) -- (4.4,0) node[right]  {$l_1^*$} coordinate(x axis);
\draw[->] (0,-0.4) -- (0,2.9)  node[above]  {$l_2^*$} coordinate(y axis);
\end{scope}

\node[label=225:{$0$}] at (0,0) {};
\draw (1,0) node[below]{$2$};
\draw (2,0) node[below]{$4$};
\draw (3,0) node[below]{$6$};
\draw (4,0) node[below]{$8$};
\draw (0,1) node[left]{$2$};
\draw (0,2) node[left]{$4$};

\node[label=270:{$l^*=(2,2)$}] at (1,1) {};
\fill[red] (1,1) circle (3pt);
\fill[red] (0,2) circle (3pt);
\fill[red] (0.5,1.5) circle (3pt);
\fill[red] (3.5,0) circle (3pt);
\fill[red] (2.5,0.5) circle (3pt);
\fill[gray!40,nearly transparent] (-0.4,1.9) -- (-0.4,3) -- (3.1,3) -- (3.1,2.1) -- (4.4,2.1) -- (4.4,-0.1) -- (3.45,-0.1) -- (3.4,0.4) -- (2.9,0.4) -- (2.9,0.9) -- (0.9,0.9) -- (0.9,1.9) -- cycle;

\draw[red,dashed] (-0.1,1.55) -- (3.1,-0.05);

\end{tikzpicture}

\caption{An example of $J_{(l_1^*,\ldots,l_d^*)}$ as in Definition \ref{SecMRdef1}. For $(\beta_1-1)/(\beta_2-1)=2$, a given set $A\in\mathbb{D}^2$ and the corresponding $(l_1^*,l_2^*)=(2,2)$, we mark the elements in $J_{(l_1^*,\ldots,l_d^*)}$ with points. Moreover, the shaded area contains all those points $(l_1,l_2)$ such that $A\cap \mathbb{D}_{l_1}\times\mathbb{D}_{l_2}\neq\emptyset$.}
\label{SecMRfig1}
\end{figure}
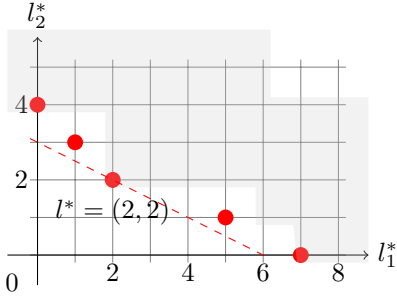

Let $\Q_\gamma(\,\cdot\,) \triangleq \P(\,\cdot\,| B_n^\gamma)$ and let $\Q_{\gamma,w}$ be as defined in \eqref{SecMReq10}, following the same strategy as in Section \ref{SecASC} we propose an estimator that takes the same form as in \eqref{SecMReq1}. Before turning to the efficiency analysis of our estimator, we summarize the findings above in Algorithm \ref{SecMRalg1}.

\begin{algorithm}
\caption{Efficient sampling of $\P\left( \bar{X}_n\in A \right)$}
\label{SecMRalg1}
\begin{algorithmic}
\renewcommand{\algorithmicrequire}{\textbf{Input:}}
\REQUIRE $\gamma\in\mathbb{R}^d$, $w\in(0,1)$
\STATE sample $u \sim \text{uniform}[0,1]$ \COMMENT{uniform distribution on $[0,1]$}
\IF{$u < w$}
\STATE sample $\bar{X}_n \sim \P\left( \bar{X}_n\in\,\cdot\, \right)$ 
\ELSE
\STATE sample $\bar{X}_n \sim \P\left( \bar{X}_n\in\,\cdot\,|\,\bar{X}_n\in B^\gamma \right)$
\ENDIF
\STATE $L\leftarrow \left[w+(1-w)\1_{B_n^\gamma} / \P(B_n^\gamma)\right]^{-1}$
\IF{$\bar{X}_n\in A$}
\RETURN $L$
\ELSE
\RETURN $0$
\ENDIF
\end{algorithmic}
\end{algorithm}

\begin{rmk}\label{SecMRrmk2}
The strategy of choosing $\gamma$ will be explained specifically via the examples presented in Section \ref{SecAAAS}, Section \ref{SecAABOP} and Section \ref{SecAAQN}.
\end{rmk}

In order to complete our algorithm, we need to discuss the computation of $\P(B_n^\gamma)$, as well as the strategy of sampling from the conditional distribution $\Q_\gamma(\,\cdot\,)$.
Since $B^\gamma$ constructed in Definition \ref{SecMRdef1} is the union of $B^{ \gamma;l }$ with $l=(l_1,\ldots,l_d)\in J_{(l_1^*,\ldots,l_d^*)}$, by the inclusion-exclusion principle, it is sufficient to discuss computing the probability of sets of the form $\bigcap_{(l_1,\ldots,l_d)\in I} B^{ \gamma;l }$, where $I$ is a finite collection of elements in $\mathbb Z_+^d$.
It turns out that the probability of such a set can be computed similarly as in Section \ref{SecASC}.
Based on this observation, we give the following proposition.

\begin{prop}\label{SecMRprop1}
The probability of $B_n^\gamma$ is equal to
\[
\sum_{k=1}^{\left| J_{(l_1^*,\ldots,l_d^*)} \right|} \left( (-1)^{k-1}\sum_{\substack{| I |=k\\I\subseteq J_{(l_1^*,\ldots,l_d^*)}}} \prod_{i=1}^d \left( 1-\exp\bigg\{- \lambda_i n\P( W^{(i)}(1)>n\gamma_i)\bigg\} \sum_{j=0}^{\hat l_{i;I}-1} \frac{(\lambda_i n)^j}{j!}\P(W^{(i)}(1)> n\gamma_i)^j \right) \right),
\]
where $\hat l_{i;I}\triangleq\max_{(l_1,\ldots,l_d)\in I}l_i$.
\end{prop}

\begin{proof}
See Section \ref{SecP}.
\end{proof}

\begin{rmk}
It should be mentioned that the complexity of computing $\P(B_n^\gamma)$ can be reduced rapidly in the case, where, for example, one is able to take a smaller (in the sense of cardinality) set than $J_{(l_1^*,\ldots,l_d^*)}$ (see e.g.\ Corollary \ref{SecMRcorol1}, Section \ref{SecAABOP} and \ref{SecAAQN} below).
\end{rmk}

As in Section \ref{SecASC}, we now discuss generating the sample path of $\bar X_n$ under $\Q_\gamma$ in the next step.
To begin with, we need the following lemma, which shows that $B^\gamma$ can be decomposed into finitely many disjoint sets.

\begin{lemma}\label{SecMRlem2}
Let $B^{\gamma;l}(i,j)\triangleq \left\{ \xi\in\mathbb D^d \,\middle|\, \#\big\{ t \,\big|\, \xi^{(i)}(t)-\xi^{(i)}(t^-)>\gamma_i \big\}\geq (l(j))_i \right\}$.
Let the elements in $J_{(l_1^*,\ldots,l_d^*)}$, denoted by $l(1),\ldots,l(| J_{(l_1^*,\ldots,l_d^*)} |)$, be ordered such that $(l(1))_d\leq\cdots\leq (l(| J_{(l_1^*,\ldots,l_d^*)} |))_d$.
Define
\begin{equation}\label{SecMReqDelta}
\Delta B^{\gamma;l}(i,j)\triangleq B^{\gamma;l}(i,j) \setminus \left( \bigcup_{m=1}^{j-1} B^{\gamma;l}(i,m) \right),\quad i\in\{1,\ldots,d-1\}.
\end{equation}
Then, we have that
\begin{align*}
B^\gamma &= \bigcup_{m_1=1}^{| J_{(l_1^*,\ldots,l_d^*)} |} \bigcup_{m_2=1}^{m_1} \cdots \bigcup_{m_{d-1}=1}^{m_{d-2}} \left( \left(\, \bigcap_{i=1}^{d-1} \Delta B^{\gamma;l}(i,m_i) \right) \cap B^{\gamma;l}(d,1) \right).
\end{align*}
\end{lemma}

\begin{proof}
See Section \ref{SecP}.
\end{proof}

Lemma \ref{SecMRlem2} shows that $B^\gamma$ can be decomposed into finitely many disjoint sets.
This implies that
\begin{align*}
\Q_\gamma(\bar X_n\in\,\cdot\,) &= \sum_{m_1=1}^{| J_{(l_1^*,\ldots,l_d^*)} |}\sum_{m_2=1}^{m_1}\cdots\sum_{m_{d-1}=1}^{m_{d-2}} h_{1;m_1,\ldots,m_{d-1}} \,\P(\bar X_n\in\,\cdot\,|\,\bar X_n\in B^\gamma(m_1,\ldots,m_{d-1})),
\end{align*}
where
\[
B^\gamma(m_1,\ldots,m_{d-1}) \triangleq \left(\, \bigcap_{i=1}^{d-1} \Delta B^{\gamma;l}(i,m_i) \right) \cap B^{\gamma;l}(d,1),
\]
and $h_{1;m_1,\ldots,m_{d-1}}\triangleq \P(\bar X_n\in B^\gamma(m_1,\ldots,m_{d-1}))/\P(\bar X_n\in B^{\gamma})$ satisfies that
\[
\sum_{m_1=1}^{| J_{(l_1^*,\ldots,l_d^*)} |}\sum_{m_2}^{m_1}\cdots\sum_{m_{d-1}}^{m_{d-2}} h_{1;m_1,\ldots,m_{d-1}}=1.
\]
Hence, it remains to design a sampling scheme for generating the sample path of $\bar X_n$ under $\P(\,\cdot\,|\, \bar X_n\in B^\gamma(m_1,\ldots,m_{d-1}))$ (for details about generating multi-dimensional discrete random numbers, see e.g.\ \cite{hucui2010}).
Due to the special structure of $B^\gamma(m_1,\ldots,m_{d-1})$, we are able to generate $\bar X_n^{(1)},\ldots,\bar X_n^{(d)}$ independently under $\P(\,\cdot\,|\, \bar X_n\in B^\gamma(m_1,\ldots,m_{d-1}))$.
To see this, first note that sampling $\bar X_n^{(d)}$ is trivial due to the discussion in Section \ref{SecASC}.
Define
\[
\check l(m_i;i)\triangleq\min_{\xi\in \Delta B^{\gamma;l}(i,m_i)} \# \{ t \,|\, \xi(t)-\xi(t^-)>\gamma_i ),
\]
and
\[
\hat l(m_i;i)\triangleq\max_{\xi\in \Delta B^{\gamma;l}(i,m_i)} \# \{ t \,|\, \xi(t)-\xi(t^-)>\gamma_i ),
\]
for $i\in\{1,\ldots,d-1\}$.
By \eqref{SecMReqDelta}, we have that
\begin{align*}
\P(\bar X_n^{(i)}\in\cdot\,|\, \bar X_n\in B^\gamma(m_1,\ldots,m_{d-1}))=\sum_{q_i=\check l(m_i;i)}^{\infty} h_{2;q_i}\,\P(\bar X_n^{(i)}\in\,\cdot \,|\, \Delta B^{\gamma;l}(i,m_i),N^{(i)}(n)=q_i),
\end{align*}
where $h_{2;q_i}\triangleq \P(\Delta B^{\gamma;l}(i,m_i),N^{(i)}(n)=q_i)/\P(\Delta B^{\gamma;l}(i,m_i))$ satisfies $\sum_{q_i\geq \check l(m_i;i)} h_{2;q_i}=1$.
Note that
\begin{align*}
\P(\Delta B^{\gamma;l}(i,m_i),N^{(i)}(n)=q_i)&=e^{-\lambda_i n}\frac{(\lambda n)^{q_i}}{q_i!}\left( \sum_{i=\check l(m_i;i)}^{\hat l(m_i;i)\wedge q_i} {q_i \choose i} \P(W^{(i)}(1)>n\gamma)^i\left(1-\P(W^{(i)}(1)>n\gamma)\right)^{q_i-i} \right).
\end{align*}
Therefore, it suffices to consider sampling $\bar X_n^{(i)}$ under $\P(\,\cdot \,|\, \Delta B^{\gamma;l}(i,m_i),N^{(i)}(n)=q_i)$. Again, we can proceed a similar approach as in Section \ref{SecASC}:
\begin{enumerate}
\item Sample $\{b_k\}_{k\leq l}$ uniformly from $\mathcal{C}\left( \left\{1,\ldots,q_i\right\},\check l(m_i;i) \right)$;
\item sample each $W^{(i)}(b_k)$, $k\leq q_i$, conditional on $W^{(i)}(1) > n \gamma_i$;
\item sample $W^{(i)}(q_i')$, $q_i'\leq q_i$, $q_i'\notin\{b_k\}_{k\leq l^*}$, under the nominal measure.
\item accept $(W^{(i)}(1),\ldots,W^{(i)}(q_i))$ with probability
\[
a( W^{(i)}(1),\ldots,W^{(i)}(q_i) ) = \binom{\#\left\{j\in\{1,\ldots,q_i\} \,\middle|\, W^{(i)}(j) > n\gamma_i \right\}}{\check l(m_i;i)}^{-1}\1_{\{\#\{j\in\{1,\ldots,q_i\} \,|\, W^{(i)}(j) > n\gamma_i \}\leq\hat l(m_i;i)\}}.
\]
\end{enumerate}

Finally, we are able to give the pseudocode of this sampling scheme in Algorithm \ref{SecMRalg2} below. 
For its expected running time, an analogous result to Proposition \ref{SecMRprop20} is formulated in Proposition \ref{SecMRprop2}.

\begin{algorithm}
\caption{Generating the sample path of $\bar X_n^{(1)},\ldots,\bar X_n^{(d)}$ under $\Q_\gamma$}
\label{SecMRalg2}
\begin{algorithmic}
\renewcommand{\algorithmicrequire}{\textbf{Input:}}

\REQUIRE $\gamma\in\mathbb R^d$

\STATE sample $(m_1,\ldots,m_{d-1}) \sim h_{1;m_1,\ldots,m_{d-1}}$

\FOR{$i=1$ to $d$}

\STATE sample $q_i \sim h_{2;q_i}$

\STATE $R\leftarrow \TRUE$

\WHILE{$R=\TRUE$}

\STATE sample $\{b_k\}_{k\leq \check l(m_i;i)} \sim \text{unif}\left( \mathcal{C}\left( \left\{1,\ldots,q_i\right\},\check l(m_i;i) \right) \right)$

\FOR{$j=1$ to $q_i$}
    \IF{$j\in \{b_k\}_{k\leq \check l(m_i;i)}$}
    \STATE sample $W^{(i)}(j)\sim W^{(i)}(1)\,\big|\,W^{(i)}(1) > n \gamma_i$
    \ELSE
    \STATE sample $W^{(i)}(j)\sim W^{(i)}(1)$
    \ENDIF
\ENDFOR

\STATE $c\leftarrow \#\left\{j\in\{1,\ldots,q_i\} \,\middle|\, W^{(i)}(j) > n \gamma_i \right\}$

\IF{$c<\hat l(m_i;i)$}

\STATE $a\leftarrow \binom{c}{\check l(m_i;i)}^{-1}$

\ELSE

\STATE $a\leftarrow 0$

\ENDIF

  \STATE sample $u\sim\text{uniform}[0,1]$
    \IF{$u<a$}
    \STATE $R\leftarrow \FALSE$
    \ELSE
    \STATE $R\leftarrow \TRUE$
    \ENDIF

\ENDWHILE

\ENDFOR

\RETURN $\bar{X}_n^{(1)},\ldots,\bar{X}_n^{(d)}$

\end{algorithmic}
\end{algorithm}

\begin{prop}\label{SecMRprop2}
Let $T_{\text{alg\ref{SecMRalg2}}}(n)$ denote the expected running time of Algorithm \ref{SecMRalg2}. Under the assumption that $W^{(i)}(1)$ is regularly varying of index $-\beta_i<-1$, for all $i\in\{1,\ldots,d\}$, we have that $T_{\text{alg\ref{SecMRalg2}}}(n)$ is uniformly bounded from above w.r.t.\ $n$, i.e.\ $\max_{n\geq0}T_{\text{alg\ref{SecMRalg2}}}(n)<\infty$.
\end{prop}

\begin{proof}
Analogous to the proof Proposition \ref{SecMRprop20}.
\end{proof}

The discussion above shows that sampling from the conditional distribution $\Q_\gamma(\,\cdot\,)$ is tractable. As we mentioned in the introduction, our estimator is straightforward to implement. Moreover, its strong efficiency, which is formulated in Theorem \ref{SecMRthm1}, can be proved based on Lemma \ref{SecMRlem1}. Without introducing any new notations, we formulate a corollary to address a special case, where it is sufficient to consider a smaller (in the sense of cardinality) set than $J_{(l_1^*,\ldots,l_d^*)}$ as in Definition \ref{SecMRdef1}.

\begin{thm}\label{SecMRthm1}
Let $B_n^\gamma \triangleq \{\bar X_n\in B^\gamma\}$, where $B^\gamma$ is as defined in \eqref{SecMReq2}. Under Assumption \ref{SecMRass2}, there exists $\gamma$ such that the estimator given by
\[
Z_n = \frac{\1_{A_n}}{w+\frac{1-w}{\P(B_n^\gamma)}\1_{B_n^\gamma}}
\]
is strongly efficient for estimating $\P(A_n)$.
\end{thm}

\begin{proof}
See Section \ref{SecP}.
\end{proof}

\begin{lemma}\label{SecMRlem1}
Let $B^\gamma$ be as defined in \eqref{SecMReq2}. Under Assumption \ref{SecMRass2}, we have that
\[\P\big(\bar{X}_n\in B^\gamma\big)=\mathcal{O}\big(\mathbb{P}(\bar{X}_n\in A)\big).\]
Moreover, there exists $\gamma$, such that
\[\P\big(\bar{X}_n\in A\cap(B^\gamma)^c\big)=\omicron\big(\mathbb{P}(\bar{X}_n\in A)^2\big).\]
\end{lemma}

\begin{proof}
See Section \ref{SecP}.
\end{proof}

\begin{corol}\label{SecMRcorol1}
Along with Assumption \ref{SecMRass2}, we assume additionally that there exists an index set $I\subseteq J_{(l_1^*,\ldots,l_d^*)}$ and $r>0$ such that: for every $\xi\in A$, there exists $(l_1,\ldots,l_d)\in I$ satisfying $d\left( \xi,\mathbb{C}_{(l_1,\ldots,l_d)} \right)\geq r$.
Set $\tilde J_{(l_1^*,\ldots,l_d^*)} = I \setminus \Delta I$, where $(l_1,\ldots,l_d)\in \Delta I$ if and only if
\begin{itemize}
\item $(l_1,\ldots,l_d)\in I$ satisfies that $\mathcal{I}(l_1,\ldots,l_d)>2\mathcal{I}(l_1^*,\ldots,l_d^*)$; and
\item for every $(l_1',\ldots,l_d')\in I\setminus\{(l_1,\ldots,l_d)\}$, we have that $\mathcal{I}(l_1,\ldots,l_d)\neq\mathcal{I}(l_1',\ldots,l_d')$.
\end{itemize}
Setting $B_n^\gamma=\{\bar X_n\in B^\gamma\}$ with
\[
B^\gamma\triangleq\bigcup_{(l_1,\ldots,l_d)\in \tilde J_{(l_1^*,\ldots,l_d^*)}} B^{ \gamma;l },
\]
there exists $\gamma$ such that the estimator given by
\[
Z_n = \frac{\1_{A_n}}{w+\frac{1-w}{\P(B_n^\gamma)}\1_{B_n^\gamma}}
\]
is strongly efficient for estimating $\P(A_n)$.
\end{corol}

\begin{proof}
Analogous to the proofs of Lemma \ref{SecMRlem1} and Theorem \ref{SecMRthm1}.
\end{proof}

\begin{rmk}\label{SecMRrmk1}
Even though our simulation algorithm is constructed in the context of Poisson processes with positive jump distributions, it can be easily generalized to the case, where the jump distributions are regularly varying at both $- \infty$ and $\infty$ (for details see the proof of Theorem 3.5 in \cite{rheeblanchetzwart2016} and the references therein).
\end{rmk}

\subsection{Extension to random walks}

Let $S_k$, $k\geq0$ be a centered random walk with increments $\{Y_k\}_{k\geq1}$. Let $\P(Y_1\leq -x)$ be regularly varying with index $-\alpha$ and let $\P(Y_1 \geq x)$ be regularly varying with index $-\beta$. Define $\bar S_n(t) = S_{\lfloor nt \rfloor}/n$, $t\geq 0$, where
\begin{equation}\label{SecMReq41}
\bar{S}_n(t)=S_{\lfloor nt \rfloor}/n,\quad t\geq0.
\end{equation}
In this subsection, we want to design an efficient simulation algorithm for estimating the probability of $\bar S_n\in A$. As in Section \ref{SecASC} and Section \ref{SecEG}, we make the following assumption for the set $A$.

\begin{assumption}\label{SecMRass3}
Let $A$ be a measurable set. Assume that $A$ is bounded away from $\mathbb{D}_{<l_-^*;l_+^*}$, where $(l_-^*,l_+^*)$ is the unique solution of the minimization problem given by \eqref{SecMReq40}.
Moreover, assume that $C_{l_-^*,l_+^*}(A^\circ)>0$.
\end{assumption}

Then, we construct the auxiliary set $B^\gamma$ as follows.

\begin{defn}\label{SecMRdef2}
Let $(l_-^*,l_+^*)$ denote the unique solution to \eqref{SecMReq40}, and let 
\begin{align*}
J_{l_-^*;l_+^*} \triangleq \left\{ (l_-,l_+)\in\mathbb Z_+^2\setminus I_{<l_-^*;l_+^*} \,\middle|\, (m_-,m_+)\prec(l_-,l_+) \text{\ implies\ } (m_-,m_+)\in I_{<l_-^*;l_+^*} \right\},
\end{align*}
where $I_{<l_-^*;l_+^*}\triangleq\left\{ (l_-,l_+)\in\mathbb Z_+^2\setminus\{(l_-^*,l_+^*)\} \,\middle|\, (\alpha-1)l_- + (\beta-1)l_+\leq(\alpha-1)l_-^* + (\beta-1)l_+^* \right\}$.
For $\gamma_->0$ and $\gamma_+$>0, define
\begin{equation}\label{SecMReq42}
B^\gamma\triangleq\bigcup_{(l_-,l_+)\in J_{l_-^*;l_+^*}} B^{ \gamma;l_-^*;l_+^* },
\end{equation}
where
$B^{ \gamma;l_-^*;l_+^* } \triangleq \left\{ \xi\in\mathbb D \,\middle|\, \#\big\{ t \,\big|\, \xi(t^-)-\xi(t)>\gamma_- \big\}\geq l_-^*,\,\#\big\{ t \,\big|\, \xi(t)-\xi(t^-)>\gamma_+ \big\}\geq l_+^* \right\}$.
\end{defn}

Defining $A_n\triangleq\{\bar S_n\in A\}$ and $B_n^\gamma\triangleq\{\bar S_n\in B^\gamma\}$, we propose an estimator for $\P(\bar S_n\in A)$ that is given by \eqref{SecMReq1}. Note that, computing $\P(\bar S_n\in B^\gamma)$, as well as generating the sample path $\bar S_n$ under $\Q^\gamma$ can be achieved by following similar strategies as in Section \ref{SecASC} and Section \ref{SecEG}. Hence, the details are omitted (for examples, see Section \ref{SecAAAS} and Section \ref{SecAABOP} below). We state the strong efficiency of our estimator in the following theorem.

\begin{thm}\label{SecMRthm3}
Let $B_n^\gamma \triangleq \{\bar X_n\in B^\gamma\}$, where $B^\gamma$ is as defined in \eqref{SecMReq42}. Under Assumption \ref{SecMRass3}, there exists $\gamma_-$ and $\gamma_+$ such that the estimator given by
\[
Z_n = \frac{\1_{A_n}}{w+\frac{1-w}{\P(B_n^\gamma)}\1_{B_n^\gamma}}
\]
is strongly efficient for estimating $\P(A_n)$.
\end{thm}

\begin{proof}
It is analogous to the proofs of Theorem \ref{SecMRthm1} and Lemma \ref{SecMRlem1}, by using Result \ref{SecPRresult4}.
\end{proof}

With the results presented in this section at hand, we are able to apply our general simulation algorithm to three examples in the next sections. These examples can be found in the applications of mathematical finance, actuarial science and queueing networks.

\section{An application to finite-time ruin probabilities}\label{SecAAAS}

\subsection*{Problem settings}
Let $S_k$, $k\geq0$, be a centered random walk with increments $\{Y_k\}_{k\geq1}$. Moreover, let $\P(Y_1 > x)$ be regularly varying at $+\infty$ with index $-\beta$. For $c\geq0$, define 
\[A_n \triangleq \left\{\max_{0\leq k\leq n}Y_k\leq nb,\,\max_{0\leq k\leq n} S_k-ck \geq na\right\}.\]
Additionally, we make a technical assumption that $a$ is not a multiple of $b$. We are interested in computing $\P(A_n)$. This probability is particularly interesting, since it is related to, for example, insurance, where huge claims may be reinsured and therefore are irrelevant in the sense of estimating the finite time ruin probability of an insurance company.

\subsection*{Large deviations results}
The rare-event probability can be estimated efficiently using the technique introduced in Section \ref{SecMR}.
To see this, define
\[
A\triangleq\left\{ \xi\in\mathbb{D}\colon \sup_{t\in[0,1]}[\xi(t)-ct]\geq a;\,\sup_{t\in[0,1]}[\xi(t)-\xi(t^-)]\leq b \right\},
\]
and $\bar{S}_n\triangleq\{\bar{S}_n(t)\}_{t\in[0,1]}$, where $\bar{S}_n(t)=S_{\lfloor nt \rfloor}/n$ for $t\geq0$. Note that $\P(A_n)=\P\left( \bar{S}_n\in A \right)$.
Set $l^* = \lceil a/b \rceil$. Intuitively, $l^*$ should be the key parameter, as it takes at least $l^*$ jumps of size $b$ to cross level $a$.
This intuition has been made rigorous by Rhee et al.\ in \cite[Section 5.1]{rheeblanchetzwart2016},
where the authors show that $A$ is bounded away from $\mathbb D_{<l^*}$, and hence, $\P(A_n) = \Theta \left(n^{l^*}\P(S_1\geq n)^{l^*}\right)$.

\subsection*{Construction of $B^\gamma$}
Since $A$ is bounded away from $\mathbb{D}_{<l^*}$, we can set
\[
B^\gamma = \left\{ \xi\in\mathbb D \,\middle|\, \#\big\{ t \,\big|\, \xi(t)-\xi(t^-)>\gamma \big\}\geq l^* \right\},
\]
and
\[
B_n^\gamma = \left\{ \bar{S}_n\in B^\gamma \right\} = \left\{ \# \left\{ k\in \{1,\ldots,n\} \,\middle|\, Y_k>n\gamma \right\} \geq l^* \right\},
\]
where $\gamma$ is the parameter that needs to be tuned. For the completeness of our algorithm, we give a closed-form expression for $\P(B_n^\gamma)$. Let $p$ denote the probability of $\P( Y_1>\gamma n )$, then we have that
\begin{equation}\label{SecAAASeq1}
\P(B_n^\gamma)=\sum_{i= l^*}^n {n\choose i}p^i\left(1-p\right)^{n-i}=1-\sum_{i=0}^{ l^*-1} {n\choose i}p^i\left(1-p\right)^{n-i},
\end{equation}
where the latter representation in \eqref{SecAAASeq1} is for numerical purposes.

\subsection*{Choice of $\gamma$}
As we mentioned in Remark \ref{SecMRrmk2}, a strategy of choosing the parameters $\gamma$ needs to be discussed in the next step. From the proof of Theorem \ref{SecMRthm1}, it is sufficient to select $\gamma$ such that $\P\left( A_n\cap (B_n^\gamma)^c \right)=\omicron \left( \P(A_n)^2 \right)$. We propose to select $\gamma$ such that $\left( a-( l^*-1)b \right)/\gamma\notin\mathbb{Z}_+$, and that
\begin{equation}\label{SecAAASeq2}
\left\lceil\frac{a-( l^*-1)b}{\gamma}\right\rceil> l^*+1.
\end{equation}
In view of Theorem \ref{SecMRthm3}, it is sufficient to show that $A\cap (B^\gamma)^c$ is bounded away from $\mathbb D_{<2l^*+1}$ with $\gamma$ satisfying \eqref{SecAAASeq2}.
To see this, choose $\theta$ with $d\left( \theta,\mathbb{D}_{<2l^*+1} \right)<r$.
This implies that there exists $\xi\in\mathbb{D}_{<2l^*+1}$ satisfying $d(\theta,\xi)<r$ and $\xi(t)=\sum_{j=1}^{2l^*} x_j\1_{[u_j,1]}(t)$.
In particular, there exists a homeomorphism $\lambda\colon[0,1]\to[0,1]$ satisfying
\begin{equation}\label{SecAAASeq3}
||\,\lambda-id\,||_{\infty} \vee ||\,\xi\circ\lambda-\theta\,||_{\infty} < r.
\end{equation}
For $\theta\in A$, using \eqref{SecAAASeq3} and the identity $\xi\circ\lambda = \theta + \left( \xi\circ\lambda-\theta \right)$, we conclude that the following holds:
\begin{enumerate}
\item $x_j<b+2r$, for every $j\in\{1,\ldots,2l^*\}$; and
\item there exists $t'$ such that
\[\sum_{u_j\leq1}x_j\geq\sum_{u_j\leq\lambda(t')}x_j>a-2r.\]
\end{enumerate}
This implies that
\[\sum_{j\geq l^*}x_j>a-2r-(l^*-1)(b+2r).\]
Moreover, for $\theta \in (B^\gamma)^c$, every jump of $\xi$ should be bounded by $\gamma+2r$ after having $l^*-1$ jumps with size bigger than $b$. Due to the fact that $\gamma$ satisfies \eqref{SecAAASeq2} and that $a$ is not a multiple of $b$, we obtain the result by choosing $r$ sufficiently small.

\subsection*{Sampling from $\Q_\gamma$}
Summarizing the discussion from previous paragraphs, we are able to propose a strongly efficient estimator for $\P(A_n)$ that is given by \eqref{SecMReq1}. As the last ingredient of our simulation algorithm, a strategy of sampling from $\Q_\gamma(\,\cdot\,)$ ($=\P(\,\cdot\,|B_n^\gamma)$) needs to be discussed. We use a similar strategy as in Algorithm \ref{SecMRalg2} and formulate the pseudocode in Algorithm \ref{SecAAASalg1}.

\begin{algorithm}
\caption{}
\label{SecAAASalg1}
\begin{algorithmic}
\STATE $R\leftarrow \TRUE$
\WHILE{$R=\TRUE$}
\STATE sample $(i_1,\ldots,i_{l^*})$ uniformly from $\mathcal{C}(\{1,\ldots,n\},l^*)$
\FOR{$j=1$ to $n$}
\IF{$j\in\{ i_1,\ldots,i_{l^*}\}$}
\STATE sample $Y_j\sim Y_1 \,|\, \gamma n < Y_1 \leq bn$
\ELSE
\STATE sample $Y_j\sim Y_1$
\ENDIF
\ENDFOR
  \STATE sample $u\sim\text{uniform}[0,1]$
  \STATE $c\leftarrow \#\left\{ m\in\{1,\ldots,n\} \,|\, \gamma n < Y_1 \leq bn \right\}$
  \STATE $a\leftarrow \binom{c}{l^*}$
  
    \IF{$u<a^{-1}$}
    \STATE $R\leftarrow \FALSE$
    \ELSE
    \STATE $R\leftarrow \TRUE$
    \ENDIF
\ENDWHILE
\end{algorithmic}
\end{algorithm}

\subsection*{Numerical results}
Finally, we investigate our algorithm numerically based on a concrete example. Let $Y_1 = Y_1'-\E Y_1'$, where $\P(Y_1'>t)=(t_r/t)^{\beta}$, i.e.\ $Y_1'$ follows a Pareto distribution with scale parameter $t_r$ and shape parameter $\beta$. Let $t_r=1$, $a=2$ and $b=1.2$. In Table \ref{sec2tab1} we select $c=0.05$, $w=0.05$, $\gamma=0.13$ and summarize the estimated probability and the level of precision (ratio between the radius of the 95\% confidence interval and the estimated value) for different combinations of $n$ and $\beta$ (based on $200000$ samples). We observe that, for different values of $\beta$, the precision stays roughly constant as $n$ grows. This confirms our theoretical results.

\begin{table}
\resizebox{\textwidth}{!}{
\centering
\begin{tabular}{lllllll}
\hline
\begin{tabular}[c]{@{}l@{}}Est\\ PR\end{tabular} & $n=1100$ & $n=1400$ & $n=1700$ & $n=2000$ & $n=2300$ & $n=2600$
\\ \hline
$\beta=1.45$&\begin{tabular}[c]{@{}l@{}}$2.188\times10^{-4}$\\$0.052$\end{tabular}&\begin{tabular}[c]{@{}l@{}}$1.691\times10^{-4}$\\$0.054$\end{tabular}&\begin{tabular}[c]{@{}l@{}}$1.376\times10^{-4}$\\$0.055$\end{tabular}&\begin{tabular}[c]{@{}l@{}}$1.204\times10^{-4}$\\$0.055$\end{tabular}&\begin{tabular}[c]{@{}l@{}}$1.076\times10^{-4}$\\$0.055$\end{tabular}&\begin{tabular}[c]{@{}l@{}}$9.603\times10^{-5}$\\$0.055$\end{tabular}
\\ \hline \hline
\begin{tabular}[c]{@{}l@{}}Est\\ PR\end{tabular} & $n=1100$ & $n=1400$ & $n=1700$ & $n=2000$ & $n=2300$ & $n=2600$\\
\hline
$\beta=1.60$&\begin{tabular}[c]{@{}l@{}}$3.243\times10^{-5}$\\$0.062$\end{tabular}&\begin{tabular}[c]{@{}l@{}}$2.335\times10^{-5}$\\$0.064$\end{tabular}&\begin{tabular}[c]{@{}l@{}}$1.509\times10^{-5}$\\$0.071$\end{tabular}&\begin{tabular}[c]{@{}l@{}}$1.324\times10^{-5}$\\$0.069$\end{tabular}&\begin{tabular}[c]{@{}l@{}}$1.150\times10^{-5}$\\$0.069$\end{tabular}&\begin{tabular}[c]{@{}l@{}}$8.809\times10^{-6}$\\$0.071$\end{tabular}
\\ \hline \hline
\begin{tabular}[c]{@{}l@{}}Est\\ PR\end{tabular} & $n=2400$ & $n=2600$ & $n=2800$ & $n=3000$ & $n=3200$ & $n=3400$
\\ \hline
$\beta=1.75$&\begin{tabular}[c]{@{}l@{}}$8.903\times10^{-7}$\\$0.095$\end{tabular}&\begin{tabular}[c]{@{}l@{}}$8.416\times10^{-7}$\\$0.092$\end{tabular}&\begin{tabular}[c]{@{}l@{}}$7.261\times10^{-7}$\\$0.094$\end{tabular}&\begin{tabular}[c]{@{}l@{}}$6.166\times10^{-7}$\\$0.096$\end{tabular}&\begin{tabular}[c]{@{}l@{}}$5.605\times10^{-7}$\\$0.096$\end{tabular}&\begin{tabular}[c]{@{}l@{}}$5.576\times10^{-7}$\\$0.093$\end{tabular}
\\ \hline \hline
\begin{tabular}[c]{@{}l@{}}Est\\ PR\end{tabular} & $n=3000$ & $n=3500$ & $n=4000$ & $n=4500$ & $n=5000$ & $n=5500$
\\ \hline
$\beta=1.90$&\begin{tabular}[c]{@{}l@{}}$5.813\times10^{-8}$\\$0.119$\end{tabular}&\begin{tabular}[c]{@{}l@{}}$3.994\times10^{-8}$\\$0.125$\end{tabular}&\begin{tabular}[c]{@{}l@{}}$3.045\times10^{-8}$\\$0.126$\end{tabular}&\begin{tabular}[c]{@{}l@{}}$2.766\times10^{-8}$\\$0.120$\end{tabular}&\begin{tabular}[c]{@{}l@{}}$2.077\times10^{-8}$\\$0.126$\end{tabular}&\begin{tabular}[c]{@{}l@{}}$1.758\times10^{-8}$\\$0.126$\end{tabular}
\\ \hline
\end{tabular}}
\caption{Estimated rare-event probability and level of precision for the application as described in Section \ref{SecAAAS} w.r.t.\ different combinations of $n$ and $\beta$.}
\label{sec2tab1}
\end{table}

\section{An application in barrier option pricing}\label{SecAABOP}
In this section we consider an application that arises in the context of financial mathematics; in particular we consider a down-in barrier option (see Section 11.3 in \cite{tankovcont2015}).

\subsection*{Problem settings}
Let $S_k$, $k\geq0$, be a centered random walk with increments $\{Y_k\}_{k\geq1}$. 
Let $\P(Y_1\leq -x)$ be regularly varying with index $-\alpha$ and let $\P(Y_1 \geq x)$ be regularly varying with index $-\beta$. 
Let $a$, $b$ and $c$ be positive real numbers. 
We provide a strongly efficient estimator for the probability of
\[
A_n \triangleq \left\{S_n \geq bn,\min_{0\leq k\leq n} S_k+c k \leq -an\right\},
\]
which can be interpreted as the chance of exercising a down-in barrier option. 
This application is interesting, since, as we will see, the large deviations behavior of $\P(A_n)$ is caused by two large jumps.

\subsection*{Large deviations results}
To begin with, define
\[A\triangleq\left\{\xi\in\mathbb{D}:\,\xi(1)\geq b,\,\inf_{0\leq t\leq1}\xi(t)+c t\leq-a\right\}.\]
Obviously, we have that $(l_-^*,l_+^*)=(1,1)$, where $(l_-^*,l_+^*)$ denotes the solution to \eqref{SecMReq40}.
To verify the topological property of $A$, we define $m,\pi_1\colon\mathbb{D}\to\mathbb{R}$ by $m(\xi)=\inf_{0\leq t\leq1} \{\xi(t)+c t\}$, and $\pi_1(\xi)=\xi(1)$.
Note that $F$, $\pi_1$ and $m$ are continuous, therefore
\[F^{-1}(A)=m^{-1}(-\infty,-a]\cap \pi_1^{-1}[b,\infty)\]
is a closed set.
By adapting the results in \cite[Section 5.2]{rheeblanchetzwart2016}, it can be shown that, for any arbitrary $i\geq0$, $\mathbb{D}_{i;0}$ and $\mathbb{D}_{0;i}$ are bounded away from $m^{-1}(-\infty,-a]$ and $\pi_1^{-1}[b,\infty)$, respectively.
Hence, $A$ is bounded away from $\mathbb D_{<1;1}$. Applying Result \ref{SecPRresult4}, we obtain that $\P(\bar X_n \in A) = \Theta\left( n^2\P(S_1\geq n)\P(S_1\leq n) \right)$.

\subsection*{Construction of $B^\gamma$}
Now we are in the framework of Theorem \ref{SecMRthm3}. Note that, by Definition \ref{SecMRdef2}, we have that $J_{1;1}=\{(1,1),(l,0),(0,m)\}$, where
\[l=\min \left\{ l'\in\mathbb{Z}_+ \,\middle|\, (l'-1)(\beta-1)>(\alpha-1) \right\},\quad m=\min \left\{ m'\in\mathbb{Z}_+ \,\middle|\, (m'-1)(\alpha-1)>(\beta-1) \right\}.\]
However, adapting the idea behind Corollary \ref{SecMRcorol1} together with the fact that $A$ is bounded away from both $\mathbb{D}_{i;0}$ and $\mathbb{D}_{0;i}$, it is sufficient to consider $\tilde J_{1;1}=\{(1,1)\}$. Hence, we can set
\[B^\gamma=\left\{ \xi\in\mathbb D \,\middle|\, \#\big\{ t \,\big|\, \xi(t^-)-\xi(t)>\gamma_- \big\}\geq 1,\,\#\big\{ t \,\big|\, \xi(t)-\xi(t^-)>\gamma_+ \big\}\geq 1 \right\}.\]
As we mentioned in the introduction, it is possible that estimators may be crafted specifically for the events of interest, in order to obtain (up to constant factors) better performance.
Due to the fact that at least one downward jump should happen \textbf{\textit{before}} upward jumps, without introducing new notations, we can modify $B^\gamma$ such that
\[B^\gamma=\left\{ \xi\in\mathbb D \,\middle|\, \exists\,t_1<t_2:\,\xi(t_1^-)-\xi(t_1)>\gamma_-,\, \xi(t_2)-\xi(t_2^-)>\gamma_+ \right\}.\]
This implies that $B_n^\gamma=\big\{\exists\,i<j:\,Y_i<-\gamma_- n,\, Y_j>\gamma_+ n \big\}$.
By a straightforward computation, we obatin that
\[
\P\left( B_n^\gamma \right) = 1-\frac{p_2}{p_2-p_1}\,(1-p_1)^n+\frac{p_1}{p_2-p_1}\,(1-p_2)^n,
\]
where $p_1\triangleq\P(Y_1 > \gamma_+ n)$ and $p_2\triangleq\P(Y_1 < -\gamma_- n)$.

\subsection*{Choice of $\gamma_-$ and $\gamma_+$}
We discuss here the strategy of choosing the parameters $\gamma_-$ and $\gamma_+$. From the proof of Theorem \ref{SecMRthm1}, it is sufficient to select $\gamma_-$, $\gamma_+$ such that $\P\left( A_n\cap (B_n^\gamma)^c \right)=\omicron \left( \P(A_n)^2 \right)$.
Hence, we propose to choose $\gamma_-$ and $\gamma_+$ such that $\left( (a+b)/\gamma_+,a/\gamma_- \right)\notin\mathbb{Z}_+^2$, and that
\begin{equation}\label{SecAABOPeq3}
\min\left\{(\alpha-1)+\left\lceil\frac{a+b}{\gamma_+}\right\rceil(\beta-1),\left\lceil\frac{a}{\gamma_-}\right\rceil(\alpha-1)+(\beta-1)\right\}>2(\alpha+\beta-2).
\end{equation}
W.l.o.g.\ we assume that $\lceil a/\gamma_- \rceil(\alpha-1)+(\beta-1)$ is the unique minimum of \eqref{SecAABOPeq3}.
It suffices to prove that $A\cap (B^\gamma)^c$ is bounded away from $\mathbb{D}_{<\lceil a/\gamma_2 \rceil;1}$.
To show that $\bigcup_{(l_-,l_+)} \mathbb{D}_{<l_-;l_+}$ with $l_-\leq \lceil a/\gamma_2 \rceil-1$ is bounded away from $A\cap (B^\gamma)^c$, choose $\theta$ with $d\left( \theta,\bigcup_{(l_-,l_+)} \mathbb{D}_{<l_-;l_+} \right)<r$. 
This implies that there exists $\xi \in \bigcup_{(l_-,l_+)} \mathbb{D}_{<l_-;l_+}$ satisfying $d(\theta,\xi)<r$, where $\xi=\sum_{k=1}^{l_+} x_k\1_{[u_k,1]}(t) - \sum_{k=1}^{l_-} y_k\1_{[v_k,1]}(t)$. 
In particular, there exists homeomorphism $\lambda\colon[0,1]\to[0,1]$ satisfying
\begin{equation}\label{SecAABOPeq4}
||\,\lambda-id\,||_{\infty} \vee ||\,\xi\circ\lambda-\theta\,||_{\infty} < r.
\end{equation}
Using \eqref{SecAABOPeq4} and the identity $\xi\circ\lambda = \theta + \left( \xi\circ\lambda-\theta \right)$, we conclude that, for $\theta \in (B^\gamma)^c$ and $t\in[0,1]$, at least one of the following holds:
\begin{itemize}
\item $x_k \leq \gamma_++2r$, for every $u_k\geq t$; or
\item $y_k \leq \gamma_-+2r$, for every $v_k< t$.
\end{itemize}
For $\theta \in m^{-1}(-\infty,-a]$, by \eqref{SecAABOPeq4} there exists $t'$ such that
\begin{equation}\label{SecAABOPeq5}
\sum_{u_j\leq \lambda(t')} x_j - \sum_{v_j\leq \lambda(t')} y_j < -a+3r.
\end{equation}
Moreover, we can assume that $y_j\leq \gamma_- + 2r$ for $j$ satisfying $v_j\leq \lambda(t')$.
Otherwise $x_j$ is bounded by $\gamma_+ + 2r$ for $j$ satisfying $v_j > \lambda(t')$.
By choosing $r$ sufficiently small, this leads to a contradiction of that $\theta \in \pi_1^{-1}[b,\infty)$ and $\lceil a/\gamma_- \rceil(\alpha-1)+(\beta-1)$ is the minimum of \eqref{SecAABOPeq3}.
Hence, \eqref{SecAABOPeq5} implies that
\[
\left( \left\lceil \frac{a}{\gamma_-} \right\rceil-1 \right) (\gamma_-+2r) > a-3r.
\]
Since $\left( \lceil a/\gamma_- \rceil-1 \right)\gamma_-<a$, choosing $r$ sufficiently small we obtain the result.
Similarly, it can be can shown that $A\cap (B^\gamma)^c$ is bounded away from $\bigcup_{(l_-,l_+)} \mathbb{D}_{<l_-;l_+}$ for $l_+\leq \lceil (a+b)/\gamma_+ \rceil-1$.

\subsection*{Sampling from $\Q_\gamma$}
Summarizing the discussion in the previous paragraphs, we are able to propose a strongly efficient estimator for $\P(A_n)$ that is given by \eqref{SecMReq1}. As in Section \ref{SecAAAS}, a strategy of sampling from $\Q_\gamma(\,\cdot\,)$ needs to be discussed. Even though $B^\gamma$ is modified to obtain smaller relative error, a similar strategy as in Algorithm \ref{SecMRalg2} can be used here. We formulate the pseudocode in Algorithm \ref{SecAABOPalg1}.

\begin{algorithm}
\caption{Sampling from $\Q_\gamma$: a modification of Algorithm \ref{SecMRalg2}}
\label{SecAABOPalg1}
\begin{algorithmic}
\renewcommand{\algorithmicrequire}{\textbf{Input:}}

\REQUIRE $\gamma_-$, $\gamma_+$
\STATE $R\leftarrow \TRUE$
\WHILE{$R=\TRUE$}
\STATE sample $(i_1,i_2)$ uniformly from $\mathcal{P}(\{1,\ldots,n\},2)$
\FOR{$j=1$ to $n$}
\IF{$j\neq i_1$ \AND $j\neq i_2$}
\STATE sample $Y_j\sim Y_1$
\ELSIF{$j=i_1$}
\STATE sample $Y_{i_1}\sim Y_1 \,|\, Y_1 < -\gamma_- n$
\ELSE
\STATE sample $Y_{i_2}\sim Y_2 \,|\, Y_2 > \gamma_+ n$
\ENDIF
\ENDFOR
  \STATE sample $u\sim\text{uniform}[0,1]$
  \IF{$(Y_1,\ldots,Y_n)\in B_n^\gamma$}
  \STATE $a\leftarrow \left( \#\{ (l,m)\in\{1,\ldots,n\}^2 \,|\, Y_l < -\gamma_- n,\,Y_m > \gamma_+ n \} \right)^{-1}$
  \ELSE
  \STATE $a\leftarrow0$
  \ENDIF
    \IF{$u<a$}
    \STATE $R\leftarrow \FALSE$
    \ELSE
    \STATE $R\leftarrow \TRUE$
    \ENDIF
\ENDWHILE
\end{algorithmic}
\end{algorithm}

\subsection*{Numerical results}
We end this section with a numerical investigation (based on $200000$ samples).
Let $Y_1 = Y_1'-\E Y_1'$, where $Y_1'$ is a random variable with density function $f_Y$ that is given by
\[
f_Y=p_1 \left(\frac{t_r}{y}\right)^\beta \1_{(t_r,\infty)}(y) + p_2 \left(\frac{t_l}{y}\right)^\alpha \1_{(-\infty,t_l)}(y) + (1-p_1-p_2)\,\frac{1}{t_r-t_l}\1_{[t_l,t_r]}(y),
\]
where $t_r>0$ and $t_l<0$.
In the following example we choose $t_r=-t_l=1$, $p_1=p_2=1/3$, $a=2$ and $b=1.5$.
In Table \ref{sec1tab3} we compare the estimated rare-event probability and precision w.r.t.\ different values of $n$, $\alpha$ and $\beta$. We observe that the precision stays roughly constant as $n$ increases for different combinations of $\alpha$ and $\beta$, which suggests the strong efficiency of our estimator.

\begin{table}
\resizebox{\textwidth}{!}{
\centering
\begin{tabular}{lllllll}
\hline
\begin{tabular}[c]{@{}l@{}}Est\\ PR\end{tabular} & $n=250$ & $n=500$ & $n=750$ & $n=1000$ & $n=1250$ & $n=1500$
\\ \hline
$\alpha=2,\hspace{6.5pt}\beta=1.5$&\begin{tabular}[c]{@{}l@{}}$3.913\times10^{-7}$\\$0.043$\end{tabular}&\begin{tabular}[c]{@{}l@{}}$1.370\times10^{-7}$\\$0.043$\end{tabular}&\begin{tabular}[c]{@{}l@{}}$6.992\times10^{-8}$\\$0.044$\end{tabular}&\begin{tabular}[c]{@{}l@{}}$4.539\times10^{-8}$\\$0.044$\end{tabular}&\begin{tabular}[c]{@{}l@{}}$3.305\times10^{-8}$\\$0.044$\end{tabular}&\begin{tabular}[c]{@{}l@{}}$2.471\times10^{-8}$\\$0.044$\end{tabular}
\\ \hline
$\alpha=1.8,\beta=1.7$&\begin{tabular}[c]{@{}l@{}}$3.322\times10^{-7}$\\$0.037$\end{tabular}&\begin{tabular}[c]{@{}l@{}}$1.154\times10^{-7}$\\$0.037$\end{tabular}&\begin{tabular}[c]{@{}l@{}}$6.040\times10^{-8}$\\$0.038$\end{tabular}&\begin{tabular}[c]{@{}l@{}}$3.840\times10^{-8}$\\$0.038$\end{tabular}&\begin{tabular}[c]{@{}l@{}}$2.870\times10^{-8}$\\$0.038$\end{tabular}&\begin{tabular}[c]{@{}l@{}}$2.225\times10^{-8}$\\$0.037$\end{tabular}
\\ \hline
$\alpha=2.3,\beta=2$&\begin{tabular}[c]{@{}l@{}}$1.923\times10^{-9}$\\$0.053$\end{tabular}&\begin{tabular}[c]{@{}l@{}}$4.004\times10^{-10}$\\$0.053$\end{tabular}&\begin{tabular}[c]{@{}l@{}}$1.491\times10^{-10}$\\$0.054$\end{tabular}&\begin{tabular}[c]{@{}l@{}}$7.601\times10^{-11}$\\$0.054$\end{tabular}&\begin{tabular}[c]{@{}l@{}}$4.632\times10^{-11}$\\$0.054$\end{tabular}&\begin{tabular}[c]{@{}l@{}}$3.072\times10^{-11}$\\$0.054$\end{tabular}
\\ \hline
$\alpha=2.7,\beta=1.8$&\begin{tabular}[c]{@{}l@{}}$6.838\times10^{-10}$\\$0.068$\end{tabular}&\begin{tabular}[c]{@{}l@{}}$1.121\times10^{-10}$\\$0.070$\end{tabular}&\begin{tabular}[c]{@{}l@{}}$4.092\times10^{-11}$\\$0.070$\end{tabular}&\begin{tabular}[c]{@{}l@{}}$2.079\times10^{-11}$\\$0.069$\end{tabular}&\begin{tabular}[c]{@{}l@{}}$1.105\times10^{-11}$\\$0.071$\end{tabular}&\begin{tabular}[c]{@{}l@{}}$6.896\times10^{-12}$\\$0.071$\end{tabular}
\\ \hline
\end{tabular}}
\caption{Estimated rare-event probability and level of precision for the application as described in Section \ref{SecAABOP} w.r.t.\ different combinations of $n$, $\alpha$ and $\beta$.}
\label{sec1tab3}
\end{table}

\section{An application to queueing networks}\label{SecAAQN}
In this section, an application to queueing networks is considered. More specifically, the probability of the number of customers in a subset of the system crossing a high level is estimated. Although some particular cases exist that allow for an explicit analysis (see e.g.\ Section 13 in \cite{debickimandjes2015}), it is hard to come up with exact results for the distribution of the workload process in general. Hence, implementing our algorithm in such a context is particularly interesting.

\subsection{Model description and preliminaries}
To be specific, we consider a $d$-dimensional stochastic fluid model. 
Suppose that jobs arrive to the $i$-th station in the network according to a Poisson process with unit rate, which is denoted by $\{N^{(i)}(t)\}_{t\geq0}$ and independent of $\{N^{(j)}\}$ for $j\neq i$. 
Moreover, the $k$-th arrival of the $i$-th station brings a job of size $W^{(i)}(k)$. 
We are assuming that $\{W(k)\triangleq(W_1(k),\ldots,W_d(k))^T\}_{k\geq1}$ is a sequence of i.i.d.\ positive random vectors and that $\{W(k)\}_{k\geq1}$ is independent of $\{N(t)\}_{t\geq0}$. 
Therefore, the total amount of external work that arrives to the $i$-th station up to time $t$ is given by $J^{(i)}(t)=\sum_{k=1}^{N^{(i)}(t)}W^{(i)}(k)$. 
Now, assume that the workload at the $i$-th station is processed as a fluid by the server at a rate $r_i$ and that a proportion $Q_{ij}\geq0$ of the fluid processed by the $i$-th station is routed to the $j$-th server. 
Moreover, we assume that $Q$ is a substochastic matrix with $Q_{ii}=0$ and that $Q^n\to0$ as $n\to\infty$. 
The dynamics of the model are expressed formally by the so-called Skorokhod map (for details see e.g.\ \cite{skorokhod1962a}, \cite{skorokhod1962b}, \cite{harrisonreiman1981} etc.), that is defined in terms of a pair of processes $(Z,Y)$ satisfying a stochastic differential equation that we shall describe now. 
Let $R=(I-Q)^T$, $r=(r_1,\ldots,r_d)^T$, $X(t)\triangleq J(t)-Rrt$ and $Z^{(i)}(t)$ denote the workload of the $i$-th station at time $t$, for given $Z^{(i)}(0)$, we have that
\begin{equation}\label{SecAAQNeq1}
dZ(t)=dX(t)+RdY(t),
\end{equation}
where $Y(\cdot)$ encodes the minimal amount of pushing required to keep $Z(\cdot)$ non-negative. 
In order to describe how to characterize the solution $(Z,Y)$ to \eqref{SecAAQNeq1}, we need to introduce some notations.
Let $\psi:\,\mathbb{D}^d\to\mathbb{D}_\uparrow^d$ with
\[\psi(x)\triangleq\inf\left\{w\in\mathbb{D}_\uparrow^d \,\middle|\, \,x+Rw\geq0\right\},\]
i.e.,
\[\psi^{(i)}(x)(t)\triangleq\inf\left\{w^{(i)}(t)\in\mathbb{R} \,\middle|\, w\in\mathbb{D}_\uparrow^d,\,x+Rw\geq0\right\},\quad\text{for all\ }i\text{\ and\ }t,\]
and $\phi:\,\mathbb{D}^d\to\mathbb{D}^d$ with $\phi(x)\triangleq x+R\psi(x)$. 
The following results summarize useful properties and characterizations of the Skorokhod mappings $\psi$, $\phi$, as well as the workload process $Z(t)$.

\begin{result}[Theorem 14.2.1, Theorem 14.2.5 and Theorem 14.2.7 of \cite{whitt2002}]\label{SecAAQNresult1}
For all $x\in\mathbb{D}^d$, the mappings $\psi$ and $\phi$ are well-defined. 
Moreover, $\psi$ and $\phi$ are Lipschitz continuous w.r.t.\ both the uniform metric and the  Skorokhod $J_1$ metric. 
If $Y(t)\triangleq\psi(X)(t)$ and $Z(t)\triangleq\phi(X)(t)$, then $(Y(t),Z(t))$ solve the Skorokhod problem given by \eqref{SecAAQNeq1}.
\end{result}

\begin{result}[Lemma 14.3.3, Corollary 14.3.4 and Corollary 14.3.5 of \cite{whitt2002}]\label{SecAAQNresult2}
Let $x\in\mathbb{D}^d$. For the discontinuity points of $\psi(x)$ (denoted by $\text{Disc}(\psi(x))$) and $\phi(x)$, we have that $\text{Disc}(\psi(x))\cup \text{Disc}(\phi(x))=\text{Disc}(x)$.
Moreover, if $x$ has only positive jumps, then $\psi(x)$ is continuous and $\phi(x)(t)-\phi(x)(t^-)=x(t)-x(t^-)$.
\end{result}

\begin{result}[Theorem 14.2.2 of \cite{whitt2002}]\label{SecAAQNresult3}
The regulator map $y=\psi(x)$ can be characterized as the unique fixed point of the map $\pi_{x,Q}:\,\mathbb{D}_\uparrow^d\to\mathbb{D}_\uparrow^d$, which is defined by
\[\pi_{x,Q}(w)(t)=\max\left\{ 0,\sup_{s\in[0,t]}Q^Tw(s)-x(s) \right\}.\]
\end{result}

\begin{result}[Consequence of Theorem 4.1 of \cite{ramasubramanian2000}]\label{SecAAQNresult4}
Let $\Delta\in\mathbb D^d$ be a non-decreasing function such that $\Delta(0)\geq0$. Then, for $x\in\mathbb D^d$, we have that
\[\psi(x)\geq\psi(x+\Delta),\quad \phi(x)\leq\phi(x+\Delta),\]
and
\[\phi(x)(t_2)-\phi(x)(t_1)\leq\phi(x+\Delta)(t_1)-\phi(x+\Delta)(t_2),\]
for any $0\leq t_1\leq t_2\leq 1$.
\end{result}

Finally, we assume that the right tail of $W^{(i)}(1)$ is regularly varying with index $-\beta_i$ and that the stability condition holds, i.e.\ $R^{-1}\rho<r$, where $\rho\triangleq\E J(1)$. Let $\bar{Z}_n(t)\triangleq Z(nt)/n$ and $\bar{X}_n(t)\triangleq X(nt)/n$. Let $c\in\{0,1\}^d$ be a binary vector, and let $\mathcal{J}_c$ denote the index set encoded by $c$, i.e., $j\in\mathcal{J}_c$ if $c_j=1$.
Set $\bar{Z}_n(t)\triangleq Z(nt)/n$ and $\bar{X}_n(t)\triangleq X(nt)/n$.
Define $l_c:\,\mathbb{R}^d\to\mathbb{R}$ by $l_c(x)=c^T x$ and $\pi_1:\,\mathbb{D}^d\to\mathbb{R}^d$ by $\pi_1(\xi)=\xi(1)$.
Moreover, let $F\triangleq l_c\circ \pi_1\circ \phi$. We are interested in estimating the probability of $\P\left(c^T \bar{Z}_n(1)\geq a\right)$.
By Theorem 14.2.6 (iii) of \cite{whitt2002}, we have that $\bar{Z}_n=\phi(\bar{X}_n)$, and hence it holds that
\begin{equation}\label{SecAAQNeq30}
\P\left(c^T \bar{Z}_n(1)\geq a\right)=\P\left(F(\bar{X}_n)\geq a\right)=\P(\bar{X}_n\in A),
\end{equation}
where $a>0$ and $A\triangleq\left\{\xi\in\mathbb{D}:\,F(\xi)\geq a\right\}$.

\subsection{Large deviations results}
To obtain the large deviations asymptotics for the rare-event probability as in \eqref{SecAAQNeq30}, we proceed the following.
\begin{itemize}
\item To determine the tail index of the rare-event probability, we study first the optimization problem given by \eqref{SecMReq3} and transform it into a (nonstandard) knapsack problem with nonlinear constraints (see \eqref{SecAAQNeq7} and Proposition \ref{SecAAQNprop1} below).
\item Under a certain assumption (see Assumption \ref{SecAAQNass1} below), we show that $A$, as defined in \eqref{SecAAQNeq30}, is bounded away from $\mathbb D_{<(l_1^*,\ldots,l_d^*)}$, where $l_1^*,\ldots,l_d^*$ is the optimal solution to the knapsack problem derived in the first step.
\item Finally, we derive a large deviations result for $\P(\bar{X}_n\in A)$ by applying Result \ref{SecPRresult3}.
\end{itemize}

We start with the optimization problem given by \eqref{SecMReq3}. Due to the fact that $X(t)$ is in general not a compensated compound Poisson process but one with certain drift, it is convenient to consider a slightly different problem, which is given by
\begin{equation}\label{SecAAQNeq2}
\argmin_{\substack{(l_1,\ldots,l_d)\in\mathbb{Z}_+^d\\\prod_{i=1}^d\mathbb{L}_{l_i}(\mu_i)\cap A\neq\emptyset}} \mathcal{I}(l_1,\ldots,l_d),
\end{equation}
where $\mu\triangleq\E X(1)=\rho -Rr$, $r'=r-R^{-1}\rho>0$ due to the stability condition, and $\mathbb{L}_{l_i}(\mu_i)\triangleq \left\{ \xi \,\middle|\, \exists\xi'\in\mathbb{D}_{l_i}:\,\xi(t)=\xi'(t)+\mu_i t=\xi'(t)-(Rr')_i t \right\}$. Define $E_0\triangleq\left\{ (l_1,\ldots,l_d)\in\mathbb{Z}_+^d \,\middle|\, l_i=0,\,\forall i\in\mathcal{J}_c \right\}$ and $E_1\triangleq\left\{ e_i \,\middle|\, i\in\mathcal{J}_c \right\}$, where $e_i$ denotes the unit vector with entries $0$ except for the $i$-th coordinate. By Result \ref{SecAAQNresult2}, instead of \eqref{SecAAQNeq2} we can solve two separate problems that are given by
\begin{equation}\label{SecAAQNeq8}
\argmin_{\substack{(l_1,\ldots,l_d)\in E_0\\\prod_{i=1}^d\mathbb{L}_{l_i}(\mu_i)\cap A\neq\emptyset}} \mathcal{I}(l_1,\ldots,l_d),
\end{equation}
and
\begin{equation}\label{SecAAQNeq9}
\argmin_{\substack{(l_1,\ldots,l_d)\in E_1\\\prod_{i=1}^d\mathbb{L}_{l_i}(\mu_i)\cap A\neq\emptyset}} \mathcal{I}(l_1,\ldots,l_d).
\end{equation}
Note that the optimization problem given by \eqref{SecAAQNeq9} can be solved easily by considering $\min_{i\in\mathcal{J}_c}\,\beta_i-1$, therefore we focus on the optimization problem given by \eqref{SecAAQNeq8}. Let $\mathcal{J}$ be a subset of $(\mathcal{J}_c)^c$. Moreover, let $\theta\in\mathbb{D}_1$ and let $\xi\in\mathbb{D}^d$ be such that
\begin{equation}\label{SecAAQNeq4}
\xi^{(i)}(t)=
\begin{cases}
-(Rr')_i t,\,t\in[0,1], &\text{for\ }i\notin\mathcal{J},\\[5pt]
\theta^{(i)}-(Rr')_i t,\,t\in[0,1],\quad &\text{for\ }i\in\mathcal{J}.
\end{cases}
\end{equation}
A necessary and sufficient condition for the existence of $\xi\in A$ is given in the following Proposition. 

\begin{prop}\label{SecAAQNprop1}
Let $\mathcal{J}\subseteq(\mathcal{J}_c)^c$. Moreover, let $\{r^*_i\}_{i\notin\mathcal{J}}$ be such that
\begin{equation}\label{SecAAQNeq5}
r_i^*=\max\left\{ r_i'-\sum_{j\neq i}Q_{ji}r_j'+\sum_{\substack{j\neq i\\j\notin\mathcal{J}}}Q_{ji}r_j^*,0 \right\},\text{\quad for\ }i\notin\mathcal{J}.
\end{equation}
Define
\begin{equation}\label{SecAAQNeq10}
\partial_z(\mathcal{J})\triangleq\sum_{i\in\mathcal{J}_c}\left( r_i^*-r_i'+\sum_{j\neq i}Q_{ji}r_j'-\sum_{\substack{j\neq i\\j\notin\mathcal{J}}}Q_{ji}r_j^* \right).
\end{equation}
If $\partial_z(\mathcal{J})\neq a$, then there exists $\xi$ satisfying \eqref{SecAAQNeq4} and $c^T\phi(\xi)(1)\geq a$, if and only if $\partial_z(\mathcal{J})> a$. Additionally, if $\mathcal{J}_1\subseteq\mathcal{J}_2\subseteq(\mathcal{J}_c)^c$, then we have that $\partial_z(\mathcal{J}_1)\leq\partial_z(\mathcal{J}_2)$.
\end{prop}

\begin{proof}
We give here a sketch of the proof, where a detailed one can be found in Section \ref{SecP}. Note that $\partial_z(\mathcal J)$ given by \eqref{SecAAQNeq10} is the increasing rate of the subset $\mathcal I_c$ of the workload process, whose associated input process does not have any jumps but starts with sufficiently large initial value. Based on this observation, a $\xi$ can be constructed for the ``if''-part of the first statement. For the ``only if''-part, suppose that there exists a $\xi$ satisfying $c^T\phi(\xi)(1)\geq a$. By Result \ref{SecAAQNresult4}, enlarging the size of jumps in $\xi$ will preserve the fact that $c^T\phi(\xi)(1)\geq a$. Hence, we can construct a new $\xi$, such that
\begin{itemize}
\item the associated workload process $\phi(\xi)$ is piecewise linear between two neighboring discontinuity points; and
\item the increasing rate of $c^T\phi(\xi)$ is always smaller or equal than $\partial_z(\mathcal J)$ given by \eqref{SecAAQNeq10}.
\end{itemize}
\end{proof}

\begin{rmk}\label{SecAAQNrmk1}
Note that \eqref{SecAAQNeq5} can be written in a matrix notation that is given by
\[r^*=\max\left\{ \left((I-Q^T)r'\right)_{\notin\mathcal{J}}+(Q_{\notin\mathcal{J}})^Tr^*,0 \right\}=\max\left\{ \left( Rr-\rho \right)_{\notin\mathcal{J}}+(Q_{\notin\mathcal{J}})^Tr^*,0 \right\},\]
where $\left( Rr-\rho \right)_{\notin\mathcal{J}}$ and $Q_{\notin\mathcal{J}}$ denote the vector and matrix respectively with its $i$-th row and column being removed for all $i\in\mathcal{J}$. Using the Banach fixed-point theorem, we obtain that $r^*=\lim_{n\to\infty}\underline{\pi}^n(0)$, where $\underline{\pi}^n\triangleq\underline{\pi}\,\circ\,\underline{\pi}^{n-1}$ and $\underline{\pi}(x)\triangleq\max\left\{ \left( Rr-\rho \right)_{\notin\mathcal{J}}+(Q_{\notin\mathcal{J}})^Tx,0 \right\}$.
\end{rmk}

Define $E_0'\triangleq E_0\cap \left\{ (l_1,\ldots,l_d)\in\mathbb{Z}_+^d \,\middle|\, l_i\in\{0,1\},\,\forall i\notin\mathcal{J}_c \right\}$ and
\[E_{\mathcal J}\triangleq\left\{ (l_1,\ldots,l_d)\in E_0' \,\middle|\, \partial_z\left( \mathcal J_{(l_1,\ldots,l_d)} \right) > a \right\},\]
where $\partial_z\left( \mathcal J_{(l_1,\ldots,l_d)} \right)$ is as defined in \eqref{SecAAQNeq10} with $\mathcal J_{(l_1,\ldots,l_d)}$ denoting the index set encoded by $(l_1,\ldots,l_d)\in E_0'$. By Proposition \ref{SecAAQNprop1}, we conclude that the optimization problem formulated in \eqref{SecAAQNeq8} is equivalent to
\begin{equation}\label{SecAAQNeq7}
\argmin_{(l_1,\ldots,l_d)\in E_{\mathcal J}}\,\mathcal{I}(l_1,\ldots,l_d).
\end{equation}
Thanks to the last statement of Proposition \ref{SecAAQNprop1}, it is unnecessary to check every $(l_1,\ldots,l_d)\in E_{\mathcal J}$ for solving \eqref{SecAAQNeq7}. Although, the optimization problem formulated in \eqref{SecAAQNeq7} is a nonstandard knapsack problem with nonlinear constraints. In the following example, we consider a specific fluid network and illustrate how to solve the optimization problem given by \eqref{SecAAQNeq7} using Proposition \ref{SecAAQNprop1}.

\begin{example}\label{SecAAQNex1}
Consider the fluid network given by $\rho=(0.8\ 0.8\ 1)^T$, $r=(1\ 1\ 2.5)^T$ and
\[Q=\begin{bmatrix}
0 & 0.1 & 0.8\\
0.1 & 0 & 0.8\\
0 & 0 & 0
\end{bmatrix}.\]
We are interested in the probability of the rare event that the third station crosses the level $na$ at time $n$ for large $n$, i.e.\ $\mathcal{J}_c=\{3\}$. It is easy to check that the stability condition holds. By an easy computation, we obtain that $\partial_z(\{1,2\})=0.1$ and $\partial_z(\{1\})=\partial_z(\{2\})=0.02$. For $a=0.05$, the optimal solution to \eqref{SecAAQNeq7} is given by $(1,1,0)$.
\end{example}

\begin{figure}[]
\centering
\includegraphics[scale=0.55]{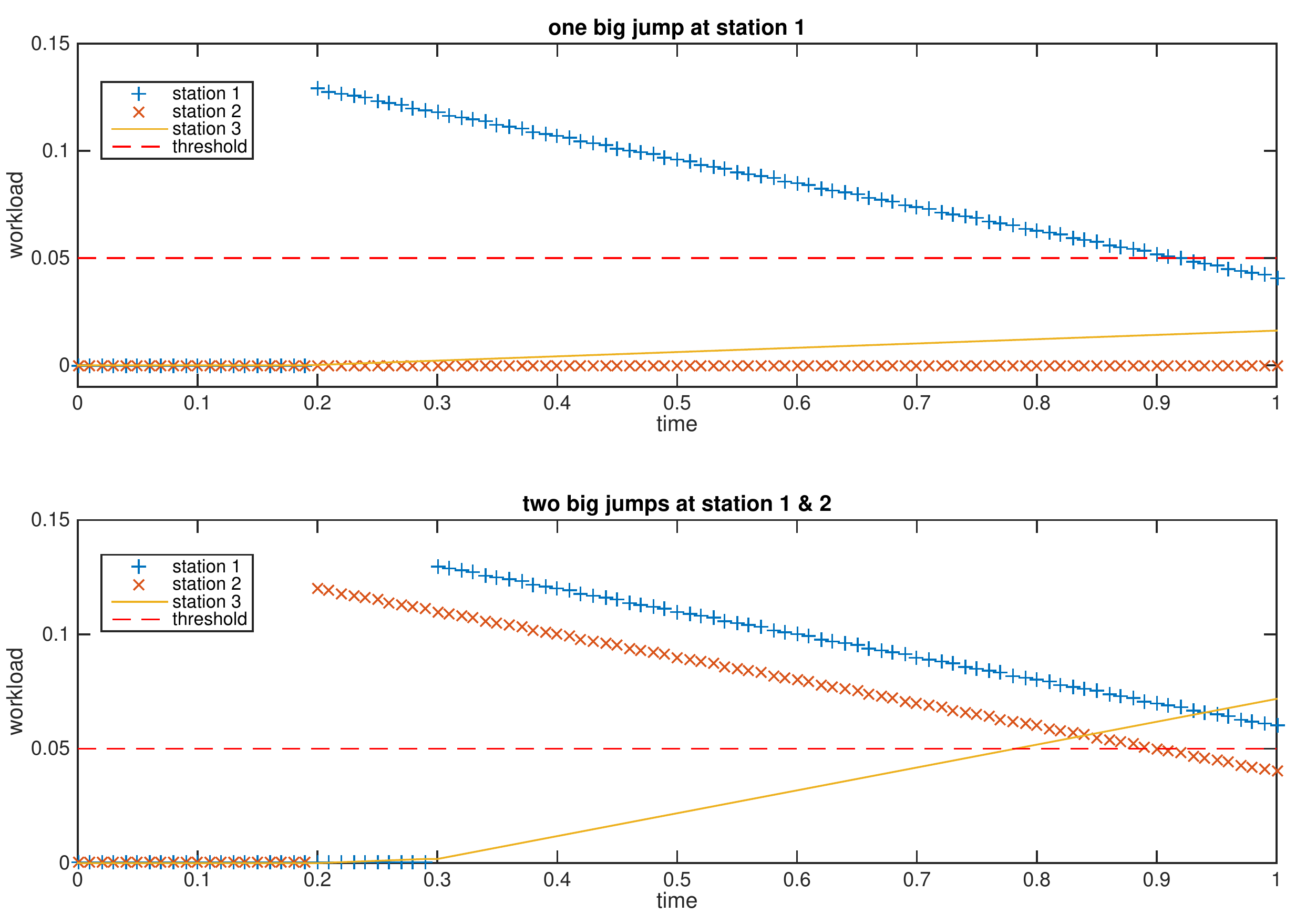}
\caption{An illustration of two different sample paths of workload processes (under the setting of Example \ref{SecAAQNex1}), whose associated input processes have the form as in \eqref{SecAAQNeq4}.}
\label{SecAAQNfig1}
\end{figure}

Suppose that we have solved the optimization problem given by \eqref{SecAAQNeq7}. To obtain the large deviations results, the following technical assumption needs to be made.

\begin{assumption}\label{SecAAQNass1}
Assume that the optimization problem formulated in \eqref{SecAAQNeq7} satisfies the conditions as follows.
\begin{description}
\item[a)] The optimization problem given by \eqref{SecAAQNeq7} has a unique solution.
\item[b)] For every $\mathcal{J}\subseteq(\mathcal{J}_c)^c$, it holds that $\partial_z(\mathcal{J})\neq a$.
\item[c)] Let $(l_1^*,\ldots,l_d^*)$ denote the optimal solution to \eqref{SecAAQNeq7}. We assume that
\[\mathcal{I}(l_1^*,\ldots,l_d^*)<\min_{i\in\mathcal{J}_c}\,\beta_i-1.\]
\end{description}
\end{assumption}

By Result \ref{SecPRresult3}, Assumption \ref{SecAAQNass1} c) implies that the rare event is caused by multiple large jumps. Throughout the rest of this section, we assume that Assumption \ref{SecAAQNass1} holds. We end this subsection with a large deviations result for $\P(\bar{X}_n\in A)=\P\left(c^T \bar{Z}_n(1)\geq a\right)$, which is formulated in the following proposition.
\begin{prop}\label{SecAAQNprop2}
Suppose that Assumption \ref{SecAAQNass1} holds. Let $F$ be as defined in \eqref{SecAAQNeq30}. Then $A=F^{-1}[a,\infty)$ is bounded away from
\[\bigcup_{(l_1,\ldots,l_d)\in I_{<(l_1^*,\ldots,l_d^*)}}\prod_{i=1}^d\mathbb{L}_{l_i}(\mu_i),\]
where $(l_1^*,\ldots,l_d^*)$ denotes the unique optimal solution of \eqref{SecAAQNeq7}. Moreover, we have that
\begin{align*}
C_{l_1^*} \times\cdots\times C_{l_d^*}\left( (F^{-1}[a,\infty))^\circ \right)
&\leq
\liminf_{n\rightarrow\infty} \frac{\P\left( \bar X_n \in A \right) }{\prod_{i=1}^d\big(n\nu_i[n,\infty)\big)^{l_i^*}}\\
&\leq\limsup_{n\rightarrow\infty} \frac{\P\left( \bar X_n \in A \right)}{\prod_{i=1}^d\big(n\nu_i[n,\infty)\big)^{l_i^*}}
\leq
C_{l_1^*} \times\cdots\times C_{l_d^*}\left( F^{-1}[a,\infty) \right).
\end{align*}
\end{prop}

\begin{proof}
See Section \ref{SecP}.
\end{proof}

\subsection{Simulation}\label{SecAAQNS}
Again, we are in the setting of Theorem \ref{SecMRthm1}.
To be able to discuss the choice of $J_{(l_1^*,\ldots,l_d^*)}$ and the parameter $\gamma$ in a more precise context, let us consider the stochastic fluid network introduced in Example \ref{SecAAQNex1}.

\addtocounter{example}{-1}
\begin{example}[continued]
Recall that, for $a=0.05$, the optimal solution of \eqref{SecAAQNeq2} is given by $\beta_1+\beta_2-2$, if we assume that $\beta_1+\beta_2-2<\beta_3-1$. Moreover, it can be easily shown that $A$ is bounded away from both $\mathbb{D}_{< i}\times\mathbb{D}_0\times\mathbb{D}_0$ and $\mathbb{D}_0\times\mathbb{D}_{<j}\times\mathbb{D}_0$. Combining this with $\mathcal{I}(1,1,1)>2\mathcal{I}(1,1,0)$, as well as Corollary \ref{SecMRcorol1}, it is sufficient to take $\tilde J_{(l_1^*,\ldots,l_d^*)}=\{(1,1,0),(0,0,1)\}$. This implies that
\begin{align*}
B_n^\gamma &= \left\{ \#\big\{ k \,\big|\, W^{(i)}(k)>n\gamma_i,\,k\leq N^{(i)}(n)\big\}\geq1,\forall i\in\{1,2\} \right\}\cup \left\{ \#\big\{ k \,\big|\, W^{(3)}(k)>n\gamma_3,\,k\leq N^{(3)}(n) \big\}\geq1 \right\},
\end{align*}
and hence
\[
(B_n^\gamma)^c = \left\{ \exists i\in\{1,2\}:\, W^{(i)}(k)\leq n\gamma_i,\,\forall 1\leq k\leq N^{(i)}(n) \right\} \cap \left\{ W^{(3)}(k)\leq n\gamma_3,\,\forall 1\leq k\leq N^{(3)}(n) \right\}.
\]

We choose $\gamma$ such that $\P(A_n\cap(B_n^\gamma)^c)=\omicron(\P(A_n)^2)$. To begin with, we assume w.l.o.g.\ that $\beta_3-1\leq2(\beta_1+\beta_2-2)$, otherwise we can simply set $J_{(l_1^*,\ldots,l_d^*)}=\{(1,1,0)\}$, since $\mathcal{I}(0,0,1)>2\mathcal{I}(1,1,0)$. Now the parameter $\gamma_3$ can be chosen such that
\[\left\lceil \frac{1/20}{\gamma_3} \right\rceil (\beta_3-1) > 2(\beta_1+\beta_2-2).\]
For the choice of $\gamma_1$, we observe that the job arriving at the second station can have arbitrarily large size. Hence, it is sufficient to consider the inequality $\partial_z(\{1,2\})t'+\partial_z(\{2\})(1-t')>a$, where $\partial_z(\{1,2\})=0.1$ and $\partial_z(\{2\})=0.02$. Solving the inequality we obtain that $t'<3/8$. This simply means that the workload process of the third station cannot exceed the level $a$ at time $1$ if we keep both of the first and the second stations overloaded less than $3/8$ of the time. Since the workload process of the first station decays at rate $1/10$, one can choose $\gamma_1$ such that
\[\left\lceil \frac{3/80}{\gamma_1} \right\rceil (\beta_1-1)+(\beta_2-1) > 2(\beta_1+\beta_2-2).\]
Analogously, it is sufficient to set $\gamma_2$ such that
\[(\beta_1-1)+\left\lceil \frac{3/80}{\gamma_2} \right\rceil (\beta_2-1) > 2(\beta_1+\beta_2-2).\]

We give a closed-form expression for $\P(B_n^\gamma)$. By assumption $\{W^{(i)}(k)\}_{1\leq i\leq d}$ are mutually independent, therefore we have that
\begin{align*}
\P((B_n^\gamma)^c)&=\P\left( \exists i\in\{1,2\}:\, W^{(i)}(k)\leq n\gamma_i,\,\forall 1\leq k\leq N^{(i)}(n) \right)\P\left(W^{(3)}(k)\leq n\gamma_3,\,\forall 1\leq k\leq N^{(3)}(n)\right)\\
&=\left[1-\prod_{i=1}^2 \left(1-\P\left( W^{(i)}(k)\leq n\gamma_i,\,\forall 1\leq k\leq N^{(i)}(n) \right)\right)\right]\P\left(W^{(3)}(k)\leq n\gamma_3,\,\forall 1\leq k\leq N^{(3)}(n)\right).
\end{align*}
Conditional on $N^{(i)}(n)$, we obtain that
\begin{align*}
\P\left( W^{(i)}(k)\leq n\gamma_i,\,\forall 1\leq k\leq N^{(i)}(n) \right)&=e^{-n}\sum_{m=0}^\infty \frac{n^m}{m!}\, \P\left( W^{(i)}(1)\leq n\gamma_i \right)^m\\
&=\exp{\left\{-n\left(1-\P\left( W^{(i)}(1)\leq n\gamma_i \right)\right)\right\}}.
\end{align*}
Summarizing the findings from above, we are able to propose a strongly efficient estimator for $\P(A_n)$ that is given by
\[
Z_n = \frac{\1_{A_n}}{w+\frac{1-w}{\P(B_n^\gamma)}\1_{B_n^\gamma}}.
\]
Moreover, Algorithm \ref{SecMRalg2} can be used to sample from $\Q^\gamma$. To see this, we decompose $B_n^\gamma$ into two disjoint sets $B_n^\gamma(1)$ and $B_n^\gamma(2)$ that are given by
\[
B_n^\gamma(1) \triangleq \left\{ \#\big\{ k \,\big|\, W^{(3)}(k)>n\gamma_3,\,k\leq N^{(3)}(n) \big\}\geq1 \right\},
\]
and
\begin{align*}
B_n^\gamma(2) 
& \triangleq \left\{ \#\big\{W^{(i)}(k)>n\gamma_i,\,k\leq N^{(i)}(n)\big\}\geq1,\forall i\in\{1,2\} \right\} \cap \{ W^{(3)}(k)\leq n\gamma_3,\,\forall 1\leq k\leq N^{(3)}(n) \},
\end{align*}
respectively.
Using Algorithm \ref{SecMRalg2}, the sample path of $\bar X_n^{(1)},\bar X_n^{(2)},\bar X_n^{(3)}$ can be simulated independently on both $B_n^\gamma(1)$ and $B_n^\gamma(2)$.
We present the numerical results based on $20000$ samples in Table \ref{SecAAQNtab1}.
We choose $W^{(i)}(1)$ such that $\P(W^{(i)}(1)>t) = (t_{r,i}/t)^{\beta_i}$ and $t_{r,i}=\rho_i(\beta_i-1)/\beta_i$, for $i\in\{1,2,3\}$.
As one can see, the numerical results suggest again what our theory predicts.
\end{example}

\begin{table}
\resizebox{\textwidth}{!}{
\centering
\begin{tabular}{lllll}
\hline
\begin{tabular}[c]{@{}l@{}}Est\\ PR\end{tabular} & $n=1200$ & $n=1600$ & $n=2000$ & $n=2400$
\\ \hline
$\beta_1=1.5,\beta_2=1.5,\beta_3=2.2$&\begin{tabular}[c]{@{}l@{}}$7.719\times10^{-2}$\\$0.045$\end{tabular}&\begin{tabular}[c]{@{}l@{}}$6.228\times10^{-2}$\\$0.058$\end{tabular}&\begin{tabular}[c]{@{}l@{}}$4.541\times10^{-2}$\\$0.057$\end{tabular}&\begin{tabular}[c]{@{}l@{}}$3.973\times10^{-2}$\\$0.057$\end{tabular}
\\ \hline \hline
\begin{tabular}[c]{@{}l@{}}Est\\ PR\end{tabular} & $n=800$ & $n=1200$ & $n=1600$ & $n=2000$ 
\\ \hline
$\beta_1=2.5,\beta_2=2.3,\beta_3=4$&\begin{tabular}[c]{@{}l@{}}$2.894\times10^{-2}$\\$0.325$\end{tabular}&\begin{tabular}[c]{@{}l@{}}$1.686\times10^{-2}$\\$0.404$\end{tabular}&\begin{tabular}[c]{@{}l@{}}$6.153\times10^{-3}$\\$0.445$\end{tabular}&\begin{tabular}[c]{@{}l@{}}$2.023\times10^{-3}$\\$0.448$\end{tabular}
\\ \hline \hline
\begin{tabular}[c]{@{}l@{}}Est\\ PR\end{tabular} & $n=600$ & $n=1000$ & $n=1400$ & $n=1800$ 
\\ \hline
$\beta_1=2.2,\beta_2=2.9,\beta_3=4.5$&\begin{tabular}[c]{@{}l@{}}$5.139\times10^{-2}$\\$0.249$\end{tabular}&\begin{tabular}[c]{@{}l@{}}$1.858\times10^{-2}$\\$0.347$\end{tabular}&\begin{tabular}[c]{@{}l@{}}$9.987\times10^{-3}$\\$0.351$\end{tabular}&\begin{tabular}[c]{@{}l@{}}$1.028\times10^{-3}$\\$0.377$\end{tabular}
\\ \hline
\end{tabular}}
\caption{Estimated rare-event probability and level of precision for the application as described in Section \ref{SecAAQNS} w.r.t.\ different combinations of $n$ and $\beta_1,\beta_2,\beta_3$.}
\label{SecAAQNtab1}
\end{table}

\section{Proofs}\label{SecP}
In this section we provide proofs of the results presented in this paper.

\begin{proof}[Proof of Proposition \ref{SecMRprop20}]
Recall that the expected running time of the rejection sampling (see Algorithm \ref{SecMRalg20} above), which is used to generate the jumps of $\bar X_n$, is bounded from above by
\[M_l=\frac{\binom{l}{l^*}\P(W(1)> n\gamma)^{l^*}}{\P(B_n^\gamma|N(n)=l)}.\]
Hence, for the expected running time of Algorithm \ref{SecMRalg20}, denoted by $T_{\text{alg\ref{SecMRalg20}}}(n)$, we have that
\begin{align*}
T_{\text{alg\ref{SecMRalg20}}}(n)&=\sum_{l\geq l^*} h_l M_l=\P(B_n^\gamma)^{-1} \sum_{l\geq l^*}\P(B_n^\gamma\,|\,N(n)=l)\P(N(n)=l)M_l\\
&=\P(B_n^\gamma)^{-1} \sum_{l\geq l^*}\P(N(n)=l)\binom{l}{l^*}\P(W(1)> n\gamma)^{l^*}\\
&=C(l^*)\,\frac{n^{l^*}(\lambda\P(W(1)> n\gamma))^{l^*}}{\P(B_n^\gamma)} e^{-\lambda n}\sum_{l\geq l^*}\frac{(\lambda n)^{l-l^*}}{(l-l^*)!}\\
&=C(l^*)\,\frac{n^{l^*}(\lambda\P(W(1)> n\gamma))^{l^*}}{\P(B_n^\gamma)}.
\end{align*}
Using Result \ref{SecPRresult2}, we obtain that $\limsup_{n\to\infty}T_{\text{alg\ref{SecMRalg20}}}(n)<\infty$.
\end{proof}

\begin{proof}[Proof of Proposition \ref{SecMRprop1}]
Let $B^\gamma$ be as defined in \eqref{SecMReq2}. For $I\subseteq J_{(l_1^*,\ldots,l_d^*)}$, define
\[
B_I^{\gamma;l}\triangleq\bigcap_{(l_1,\ldots,l_d)\in I} B^{ \gamma;l }.
\]
By the inclusion-exclusion principle, we have that
\begin{equation}\label{SecPeq40}
\P(\bar{X}_n\in B^\gamma)=\sum_{k=1}^{\left| J_{(l_1^*,\ldots,l_d^*)} \right|} \left( (-1)^{k-1}\sum_{\substack{| I |=k\\I\subseteq J_{(l_1^*,\ldots,l_d^*)}}} \P\left( \bar{X}_n\in B_I^{ \gamma;l } \right) \right).
\end{equation}
Moreover, for any finite collection $I$ of elements in $\mathbb{Z}_+^d$ with $I\subseteq J_{(l_1^*,\ldots,l_d^*)}$, we have that
\begin{align}
\nonumber B_I^{\gamma;l}&=\bigcap_{i=1}^d\, \bigcap_{(l_1,\ldots,l_d)\in I}\, \left\{ \left( \xi^{(1)},\ldots,\xi^{(d)} \right) \,\middle|\, \#\big\{ t \,\big|\, \xi^{(i)}(t)-\xi^{(i)}(t^-)>\gamma_i \big\}\geq l_i \right\}\\
&=\bigcap_{i=1}^d\, \left\{ \left( \xi^{(1)},\ldots,\xi^{(d)} \right) \,\middle|\, \#\big\{ t \,\big|\, \xi^{(i)}(t)-\xi^{(i)}(t^-)>\gamma_i \big\}\geq \hat l_{i;I} \right\},\label{SecPeq4}
\end{align}
where $\hat l_{i;I}\triangleq\max_{(l_1,\ldots,l_d)\in I}l_i$. Since $\bar X_n^{(1)},\ldots,\bar X_n^{(d)}$ are independent processes, we obtain that
\begin{align*}
\P(B_I^{\gamma;l})=\prod_{i=1}^d \left( 1-\exp\bigg\{- \lambda_i n\P( W^{(i)}(1)>n\gamma_i)\bigg\} \sum_{j=0}^{\hat l_{i;I}-1} \frac{(\lambda_i n)^j}{j!}\P(W^{(i)}(1)> n\gamma_i)^j \right).
\end{align*}
\end{proof}

\begin{proof}[Proof of Lemma \ref{SecMRlem2}]
Recall that
\[
B^{\gamma;l}(i,j)\triangleq \left\{ \xi\in\mathbb D^d \,\middle|\, \#\big\{ t \,\big|\, \xi^{(i)}(t)-\xi^{(i)}(t^-)>\gamma_i \big\}\geq (l(j))_i \right\}.
\]
Hence, we have that
\begin{align}
B^\gamma = \bigcup_{j=1}^{| J_{(l_1^*,\ldots,l_d^*)} |} \bigcap_{i=1}^d B^{\gamma;l}(i,j) = \bigcup_{j=1}^{| J_{(l_1^*,\ldots,l_d^*)} |} \left( B^{\gamma;l}(1,j) \cap \bigcap_{i=2}^d B^{\gamma;l}(i,j) \right).\label{SecPeq341}
\end{align}
By definition
\[
\Delta B^{\gamma;l}(i,j)\triangleq B^{\gamma;l}(i,j) \setminus \left( \bigcup_{m=1}^{j-1} B^{\gamma;l}(i,m) \right).
\]
Therefore, we have that
\begin{equation}\label{SecPeq342}
B^{\gamma;l}(i,j)=\bigcup_{m_i=1}^{j} \Delta B^{\gamma;l}(i,j).
\end{equation}
Plugging \eqref{SecPeq342} into \eqref{SecPeq341}, we obtain that
\begin{align*}
B^\gamma &= \bigcup_{j=1}^{| J_{(l_1^*,\ldots,l_d^*)} |} \left( \left( \bigcup_{m_1=1}^{j} \Delta B^{\gamma;l}(1,m_1) \right) \cap \bigcap_{i=2}^d B^{\gamma;l}(i,j) \right)\\
&= \bigcup_{m_1=1}^{| J_{(l_1^*,\ldots,l_d^*)} |} \left(\, \bigcup_{j=m_1}^{| J_{(l_1^*,\ldots,l_d^*)} |} \left( \Delta B^{\gamma;l}(1,m_1) \cap \bigcap_{i=2}^d B^{\gamma;l}(i,j) \right) \right)\\
&= \bigcup_{m_1=1}^{| J_{(l_1^*,\ldots,l_d^*)} |} \left( \Delta B^{\gamma;l}(1,m_1) \cap \left( \bigcup_{j=1}^{m_1} \bigcap_{i=2}^d B^{\gamma;l}(i,j) \right) \right).
\end{align*}
Applying the same procedure to $\bigcup_{j=1}^{m_1} \bigcap_{i=2}^d B^{\gamma;l}(i,j)$, we obtain that
\begin{align*}
B^\gamma &= \bigcup_{m_1=1}^{| J_{(l_1^*,\ldots,l_d^*)} |} \bigcup_{m_2=1}^{m_1} \left( \Delta B^{\gamma;l}(1,m_1) \cap \Delta B^{\gamma;l}(2,m_2) \cap \left( \bigcup_{j=1}^{m_2} \bigcap_{i=3}^d B^{\gamma;l}(i,j) \right) \right).
\end{align*}
Iterating the same procedure $d-1$ times, we obtain that
\begin{align}
B^\gamma &= \bigcup_{m_1=1}^{| J_{(l_1^*,\ldots,l_d^*)} |} \bigcup_{m_2=1}^{m_1} \cdots \bigcup_{m_{d-1}=1}^{m_{d-2}} \left( \left(\, \bigcap_{i=1}^{d-1} \Delta B^{\gamma;l}(i,m_i) \right) \cap \left( \bigcup_{j=1}^{m_{d-1}} B^{\gamma;l}(d,j) \right) \right).\label{SecPeq343}
\end{align}
Since $l(1),\ldots,l(| J_{(l_1^*,\ldots,l_d^*)} |)$ are ordered such that $(l(1))_d\leq (l(2))_d\leq\cdots\leq (l(| J_{(l_1^*,\ldots,l_d^*)} |))_d$, we obtain that
\begin{equation}\label{SecPeq344}
\bigcup_{j=1}^{m_{d-1}} B^{\gamma;l}(d,j)=B^{\gamma;l}(d,1).
\end{equation}
Plugging \eqref{SecPeq344} into \eqref{SecPeq343}, we obtain that
\begin{align*}
B^\gamma &= \bigcup_{m_1=1}^{| J_{(l_1^*,\ldots,l_d^*)} |} \bigcup_{m_2=1}^{m_1} \cdots \bigcup_{m_{d-1}=1}^{m_{d-2}} \left( \left(\, \bigcap_{i=1}^{d-1} \Delta B^{\gamma;l}(i,m_i) \right) \cap B^{\gamma;l}(d,1) \right).
\end{align*}
\end{proof}

\begin{proof}[Proof of Theorem \ref{SecMRthm1}]
For the second moment of $Z$ (under the change of measure) we have that
\begin{align*}
\E^{\Q_{\gamma,w}}[Z_n^2] &= \E [Z_n] \\
&= \E\left[ Z_n \1_{B_n^\gamma} \right] + \E\left[ Z_n \1_{(B_n^\gamma)^c} \right]\\
&\leq \frac{1}{1-w}\P(A_n\cap B_n^\gamma)\P(B_n^\gamma) + \frac{1}{w} \P(A_n \cap (B_n^\gamma)^c)\\
&\leq \frac{1}{1-w}\P(A_n)\P(B_n^\gamma) + \frac{1}{w} \P(A_n \cap (B_n^\gamma)^c).\label{sec1eq1}
\end{align*}
Combining this with Lemma \ref{SecMRlem1} we obtain the strong efficiency of our estimator.\\
\end{proof}

\begin{proof}[Proof of Lemma \ref{SecMRlem1}]
First, note that $\P\big(\bar{X}_n\in B^\gamma\big)=\mathcal{O}\big(\mathbb{P}(\bar{X}_n\in A)\big)$ follows immediately from Result \ref{SecPRresult3}.

We need to show the existence of $\gamma$ such that $\P\big(\bar{X}_n\in A\cap(B^\gamma)^c\big)=\omicron(\P(\bar{X}_n\in A)^2)$. Since $A$ is bounded away from $\mathbb{D}_{<(l_1^*,\ldots,l_d^*)}$ by assumption, there exists $r$ such that $d\left( A,\mathbb{D}_{<(l_1^*,\ldots,l_d^*)} \right)\geq r$. On the one hand, from \cite{rheeblanchetzwart2016} we have that
\begin{equation}\label{SecPeq9}
A \subseteq \left\{ \left( \xi^{(1)},\ldots,\xi^{(d)} \right) \,\middle|\, \exists (l_1,\ldots,l_d)\in J_{(l_1^*,\ldots,l_d^*)}: d\left(\xi^{(i)},\mathbb{D}_{< l_i}\right)\geq r,\,\forall i\in\{1,\ldots,d\} \right\}.
\end{equation}
On the other hand, we have that
\begin{align}
\nonumber(B^\gamma)^c = \Bigg\{ \left( \xi^{(1)},\ldots,\xi^{(d)} \right) \,\Bigg|\,& \forall (l_1,\ldots,l_d)\in J_{(l_1^*,\ldots,l_d^*)}:\\
&\exists i\in\{1,\ldots,d\}: \#\big\{ t \,\big|\, \xi^{(i)}(t)-\xi^{(i)}(t^-)>\gamma_i \big\}\leq l_i-1 \Bigg\}.\label{SecPeq10}
\end{align}
Let $\xi=\left( \xi^{(1)},\ldots,\xi^{(d)} \right)\in A \cap (B^\gamma)^c$ be a step function in the set $\prod_{i=1}^d \mathbb{D}_{l_i'}$. By \eqref{SecPeq9}, there exists $(l_1,\ldots,l_d)\in J_{(l_1^*,\ldots,l_d^*)}$, such that $\xi^{(i)}=\sum_{j=1}^{l_i+m_i} c_j^{(i)}\1_{[t_j^{(i)},1]}$, $m_i\in\mathbb{Z}_+$ and $d\left(\xi^{(i)},\mathbb{D}_{< l_i}\right)\geq r$ for all $i\in\{1,\ldots,d\}$ with $l_i\neq0$. Combining $d\left(\xi^{(i)},\mathbb{D}_{< l_i'}\right)\geq r$ with the fact that $\xi^{(i)}=\sum_{j=1}^{l_i-1} c_j^{(i)}\1_{[t_j^{(i)},1]}\in\mathbb{D}_{< l_i}$, we conclude that
\begin{equation}\label{SecPeq11}
\sum_{j=l_i}^{l_i+m_i} c_j^{(i)} \geq d\left( \sum_{j=1}^{l_i+m_i} c_j^{(k)}\1_{[t_j^{(i)},1]},\,\sum_{j=1}^{l_i-1} c_j^{(i)}\1_{[t_j^{(i)},1]} \right)\geq r,
\end{equation}
or in other words, the sum of $m_i+1$ smallest jump is bounded from below by $r$ for each $\xi^{(i)}$ of $\{\xi^{(i)}\}_{i\in\{1,\ldots,d\}}$ satisfying $l_i\neq0$. Combining \eqref{SecPeq10} with \eqref{SecPeq11}, as well as choosing $\gamma_k$ sufficiently small, there exists at least one $k\in\{1,\ldots,d\}$ such that the smallest jump of $\xi^{(k)}$ is bounded from below by $r'>0$ for arbitrary but fixed $m_k$. Repeating the same procedure as described above, we can construct $(m_1,\ldots,m_d)$ for every $(l_1,\ldots,l_d)\in J_{(l_1^*,\ldots,l_d^*)}$, such that the optimization problem, given by
\begin{equation}\label{SecPeq16}
\argmin_{\substack{(l_1,\ldots,l_d)\in\mathbb{Z}_+^d\\\prod_{i=1}^d\mathbb{D}_{l_i}\cap A\cap(B^\gamma)^c\neq\emptyset}} \mathcal{I}(l_1,\ldots,l_d),
\end{equation}
has a unique solution $(l_1^{**},\ldots,l_d^{**})$ satisfying $\mathcal{I}(l_1^{**},\ldots,l_d^{**})>2\mathcal{I}(l_1^*,\ldots,l_d^*)$.
We denote this specific choice of $(m_1,\ldots,m_d)$ for every $(l_1,\ldots,l_d)\in J_{(l_1^*,\ldots,l_d^*)}$ by $\left\{ m^{(l_1,\ldots,l_d)} \right\}_{(l_1,\ldots,l_d)\in J_{(l_1^*,\ldots,l_d^*)}}$.
It should be noted that the existence and the uniqueness of $(l_1^{**},\ldots,l_d^{**})$ can be guaranteed by enlarging the set $A$ (since we are looking for an upper bound for $\P(\bar{X}_n\in A\cap(B^\gamma)^c)$), together with choosing the corresponding $\gamma_i$ sufficiently small.
Therefore, it remains to show that, under the chosen $\gamma$, the set $A\cap(B^\gamma)^c$ is bounded away from $\mathbb{D}_{<(l_1^{**},\ldots,l_d^{**})}$. Select $\xi$ satisfying $d\left( \xi,\mathbb{D}_{<(l_1^{**},\ldots,l_d^{**})} \right)<\delta$, and hence, there exists $\theta\in\mathbb{D}_{<(l_1^{**},\ldots,l_d^{**})}$ such that $d(\xi,\theta)<\delta$. On the one hand, combining $d(\xi,\theta)<\delta$ with \eqref{SecPeq9}, there exists $(l_1,\ldots,l_d)\in J_{(l_1^*,\ldots,l_d^*)}$ such that $d\left(\theta^{(i)},\mathbb{D}_{< l_i}\right) > r-\delta$, for all $i\in\{1,\ldots,d\}$. Hence, we have that $\theta^{(i)}=\sum_{j=1}^{l_i+m_i} c_j^{(i)}\1_{[t_j^{(i)},1]}$, $m_i\in\mathbb{Z}_+$, satisfying
\begin{equation}\label{SecPeq14}
\sum_{j=l_i}^{l_i+m_i} c_j^{(i)} \geq d\left( \sum_{j=1}^{l_i+m_i} c_j^{(k)}\1_{[t_j^{(i)},1]},\,\sum_{j=1}^{l_i-1} c_j^{(i)}\1_{[t_j^{(i)},1]} \right)\geq r-\delta,
\end{equation}
for all $i\in\{1,\ldots,d\}$ with $l_i\neq0$. On the other hand, there exist homeomorphisms $\{\lambda_i\}_{i\in\{1,\ldots,d\}}$ such that
\begin{equation}\label{SecPeq12}
||\,\lambda_i-id\,||_{\infty} \vee ||\,\theta^{(i)}\circ\lambda_i-\xi^{(i)}\,||_{\infty} < \delta,
\end{equation}
for $i\in\{1,\ldots,d\}$. Combining \eqref{SecPeq12} with \eqref{SecPeq10}, we conclude the existence of at least one $i\in\{1,\ldots,d\}$ such that
\begin{equation}\label{SecPeq13}
\#\big\{ t \,\big|\, \xi^{(i)}(t)-\xi^{(i)}(t^-)>\gamma_i-\delta \big\}\leq l_i-1.
\end{equation}
Since $\theta\in\mathbb{D}_{<(l_1^{**},\ldots,l_d^{**})}$, we have that
\begin{equation}\label{SecPeq15}
m_i \leq (m^{(l_1,\ldots,l_d)})_i-1,
\end{equation}
for all $i\in\{1,\ldots,d\}$ with $l_i\neq0$. Finally, by \eqref{SecPeq14}, \eqref{SecPeq13} and the choice of $\gamma$, we conclude that choosing $\delta$ sufficiently small leads to contradiction of \eqref{SecPeq15}.
\end{proof}

\begin{proof}[Proof of Proposition \ref{SecAAQNprop1}]
We derive a necessary and sufficient condition for $c^T\phi(\xi)(1)\geq a$ with $\xi$ \eqref{SecAAQNeq4}.

For the ``only if''-part, suppose that $\partial_z(\mathcal J)>a$. Let $(v_1,\ldots,v_d)\in\mathbb{R}_+^d$, $\delta\in(0,1)$ and $\xi$ be such that
\[
\xi^{(i)}(t)=
\begin{cases}
-(Rr')_i t,\,t\in[0,1], &\text{for\ }i\notin\mathcal{J},\\[5pt]
v_i\1_{[\delta,1]}(t)-(Rr')_i t,\,t\in[0,1],\quad &\text{for\ }i\in\mathcal{J}.
\end{cases}
\]
Obviously $\xi$ satisfies \eqref{SecAAQNeq4}. For $t\in[0,\delta)$, by Result \ref{SecAAQNresult3}, the regulator process $y_{\xi}\triangleq\psi(\xi)$ should satisfy the fixed point equation that is given by
\[
y_{\xi}^{(i)}(t)=\max\left\{ 0,\sup_{s\in[0,t]}\sum_{j\neq i}Q_{ji}y_{\xi}^{(j)}(s)+(Rr')_i s \right\},\text{ for all }i\in\{1,\ldots,d\}.
\]
Using the fact that $r'>0$, we obtain that $y_{\xi}(t)=r' t$, for $t\in[0,\delta)$. For $t\in[\delta,1]$, again by Result \ref{SecAAQNresult3}, it holds that
\begin{equation}\label{SecPeq17}
y_{\xi}^{(i)}(t)=\max\left\{0,-v_i+\sup_{s\in[0,t]}r_i's+\sum_{j\neq i}Q_{ji}(y_{\xi}^{(j)}(s)-r_j's)\right\},\text{\quad for all\ }i\in\mathcal{J},
\end{equation}
and
\begin{equation}\label{SecPeq18}
y_{\xi}^{(i)}(t)=\max\left\{0,\sup_{s\in[0,t]}r_i's+\sum_{\substack{j\in\mathcal{J}}}Q_{ji}(y_{\xi}^{(j)}(s)-r_j's)+\sum_{\substack{j\neq i\\j\notin\mathcal{J}}}Q_{ji}(y_{\xi}^{(j)}(s)-r_j's)\right\},\text{\quad for all\ }i\notin\mathcal{J}.
\end{equation}
Since $\{v_i\}_{i\in\mathcal{J}}$ are non-negative, by Result \ref{SecAAQNresult2}, we conclude that $y_{\xi}(s)$, as well as $r_i's+\sum_{j\neq i}Q_{ji}(y_{\xi}^{(j)}(s)-r_j's)$ are continuous in $s$ on $[0,1]$. Using the Bolzano-Weierstrass theorem, there exists a set of sufficiently large $\{v_i\}_{i\in\mathcal{J}}$ (depending on $y_\xi$), such that $y_{\xi}^{(i)}(t)=y_{\xi}^{(i)}(\delta)=r_i'\delta$ for $i\in\mathcal{J}$. Plugging this into \eqref{SecPeq18} along with setting $y_{\xi}^{(i)}(t)=r_i'\delta+r_i^*(t-\delta)$ for $i\notin\mathcal J$, $t\in[\delta,1]$, we obtain that
\begin{align}
\nonumber r_i'\delta+r_i^*(t-\delta)&=\max\left\{ 0,\,\sup_{s\in[0,t]}r_i's + \sum_{\substack{j\in\mathcal{J}}}Q_{ji}(y_{\xi}^{(j)}(s)-r_j's) + \sum_{\substack{j\neq i\\j\notin\mathcal{J}}}Q_{ji}(y_{\xi}^{(j)}(s)-r_j's) \right\}\\
&=\max\left\{ r_i'\delta,\,r_i'\delta+\max_{s\in[\delta,t]}r_i'(s-\delta) - \sum_{j\neq i}Q_{ji}r_j'(s-\delta) + \sum_{\substack{j\neq i\\j\notin\mathcal{J}}}Q_{ji}r_j^*(s-\delta) \right\},\text{\quad for all\ }i\in\mathcal{J}.\label{SecPeq20}
\end{align}
Note that \eqref{SecPeq20} is solved by $r_i^*$ satisfying \eqref{SecAAQNeq5}. Moreover, by a straightforward computation, for the workload process $z_\xi\triangleq\phi(\xi)$, we obtain that $c^Tz_\xi(1)=\partial_z(\mathcal J)(1-\delta)$. Since by assumption $\partial_z(\mathcal J)>a$, we can choose $\delta$ such that $c^Tz_\xi(1)\geq a$.

For the other direction of the proof, suppose that $c^T\phi(\xi)(1)\geq a$ for some $\xi$ satisfying \eqref{SecAAQNeq4}. Let the jump sizes and the associated jump times of $\xi$ be denoted by $\{u_i\}_{i\in\mathcal J}$ and $\{t_i\}_{i\in\mathcal J}$, respectively. First we should mention that, by Result \ref{SecAAQNresult4}, enlarging $\{u_i\}_{i\in\mathcal J}$ will preserve the fact that $c^T\phi(\xi)(1)\geq a$. Moreover, let $d_1<\cdots< d_m$ denote the discontinuity points of $\xi$ with $m\leq |\mathcal{J}|$ and define $\mathcal J_i\triangleq\{ k \,|\, t_k\leq d_i \}$, for every $i\in \{1,\ldots,m\}$. Now observe that $y_{\xi}(t)=r't$, $t\in[0,d_1)$. Hence, we have that $z_\xi'(t)=0\leq\partial_z(\mathcal J)$, for $t\in[0,d_1)$. For $y_{\xi}(t)$, $t\in[d_1,d_2)$, we can easily check that
\[
y_\xi^{(i)}(t)=
\begin{cases}r_i'd_1,&\text{for all\ }i\in\mathcal J_1,\\[5pt]
r_i'd_1+r_i^{*,1}(t-d_1),\quad &\text{for all\ }i\notin\mathcal J_1,
\end{cases}
\]
by taking sufficiently large $\{u_i\}_{i\in\mathcal J_1}$, where
\[r_i^{*,1}=\max\left\{ r_i'-\sum_{j\neq i}Q_{ji}r_j'+\sum_{\substack{j\neq i\\j\notin\mathcal{J}_1}}Q_{ji}r_j^{*,1},0 \right\},\text{\quad for\ }i\notin\mathcal{J}_1.\]
Since $\mathcal J_1\subseteq\mathcal J$, by Result \ref{SecAAQNresult4} and \eqref{SecAAQNeq5}, we conclude that $z_\xi'(t)\leq\partial_z(\mathcal J)$, for $t\in[d_1,d_2)$. Defining $\mathcal J_1'\triangleq\mathcal J_1\cup\{k \,|\, r_k^{*,1}=0 \}$, we consider $y_\xi(t)$ for $t\in[d_2,d_3)$. Following a similar argument as above, we claim that
\[
y_\xi^{(i)}(t)=
\begin{cases}
r_i'd_1, &\text{for all\ }i\in\mathcal J_1',\\[5pt]
r_i'd_1+r_i^{*,1}(d_2-d_1), &\text{for all\ }i\in\mathcal J_2 \setminus \mathcal J_1',\\[5pt]
r_i'd_1+r_i^{*,1}(d_2-d_1)+r_i^{*,2}(t-d_2),\quad &\text{for all\ }i\notin\mathcal J_1' \cup \mathcal J_2,
\end{cases}
\]
for sufficiently large $\{u_i\}_{i\in\mathcal J_1\cup\mathcal J_2}$, where
\[r_i^{*,2}=\max\left\{ r_i'-\sum_{j\neq i}Q_{ji}r_j'+\sum_{\substack{j\neq i\\j\notin\mathcal{J}_1'\cup\mathcal{J}_2}}Q_{ji}r_j^{*,2},0 \right\},\text{\quad for\ }i\notin\mathcal{J}_1'\cup\mathcal{J}_2.\]
Consider the fixed point equation that is given by
\begin{equation}\label{SecPeq50}
\tilde{r}_i^{*,2}=\max\left\{ r_i'-\sum_{j\neq i}Q_{ji}r_j'+\sum_{\substack{j\neq i\\j\notin\mathcal{J}_1\cup\mathcal{J}_2}}Q_{ji} \tilde{r}_j^{*,2},0 \right\},\text{\quad for\ }i\notin\mathcal{J}_1\cup\mathcal{J}_2.
\end{equation}
Since $\mathcal J_1\subseteq\mathcal J_1\cup \mathcal J_2$, by Result \ref{SecAAQNresult4}, we obtain that $\tilde{r}_k^{*,2}=0$, for every $k\in\mathcal J_1'\setminus\mathcal J_1$. By making the convention that $r_k^{*,2}=0$ for $k\in\mathcal J_1'\setminus\mathcal J_1$, we claim that $r_i^{*,2}=\tilde{r}_i^{*,2}$, for $i\notin\mathcal{J}_1\cup\mathcal{J}_2$. Since $\mathcal J_1\cup\mathcal J_2\subseteq\mathcal J$, by Result \ref{SecAAQNresult4}, \eqref{SecPeq50} and \eqref{SecAAQNeq5}, we conclude that $z_\xi'(t)\leq\partial_z(\mathcal J)$, for $t\in[d_2,d_3)$. Iterating the same procedure $m$ more times, we can construct a $\xi$ (by taking $\{u_i\}_{i\in\mathcal J}$ sufficiently large) such that $z_\xi$ is piecewise linear between neighboring discontinuity points. Moreover, the increasing rate of $z_\xi$ is less than $\partial_z(\mathcal J)$, i.e.\ $z_\xi'(t)\leq\partial_z(\mathcal J)$, for $t\in[0,1]$. Therefore, we obtain that $\partial_z(\mathcal J)>a$.

The last statement of Proposition \ref{SecAAQNprop1} is a consequence of Result \ref{SecAAQNresult4}.
\end{proof}

\begin{proof}[Proof of Proposition \ref{SecAAQNprop2}]
Let the unique optimal solution of \eqref{SecAAQNeq7} be denoted by $(l_1^*,\ldots,l_d^*)$. 
To prove that $A$ is bounded away from
\[
\bigcup_{(l_1,\ldots,l_d)\in\mathcal{I}_{<l_1^*,\ldots,l_d^*}}\prod_{i=1}^d\mathbb{L}_{l_i}(\mu_i),
\]
it is sufficient to show that $A=F^{-1}[a,\infty)$ is bounded away from $\prod_{i=1}^d\mathbb{L}_{l_i}(\mu_i)$ for all $(l_1,\ldots,l_d)\in\mathcal{I}_{<l_1^*,\ldots,l_d^*}$.
To begin with, let $(l_1,\ldots,l_d)\in\mathcal{I}_{<l_1^*,\ldots,l_d^*}$.
Under Assumption \ref{SecAAQNass1} we have that $\partial_z(\mathcal{I})<a$, where $j\in\mathcal{I}$ if and only if $l_j\neq0$.
Applying a similar approach as in the proof of Proposition \ref{SecAAQNprop1}, it can be shown that $F\left( \prod_{i=1}^d\mathbb{L}_{l_i}(\mu_i) \right)\subseteq(-\infty,\partial_z(\mathcal{I})]$.
This implies that there exists $\delta>0$ satisfying
\begin{equation}\label{SecPeq30}
d\left( F\left(\prod_{i=1}^d\mathbb{L}_{l_i}(\mu_i)\right),[a,\infty) \right)>\delta.
\end{equation}
Moreover, by Result \ref{SecAAQNresult1} we conclude that the mapping $F$ as composition of Lipschitz continuous mappings (for continuity of $\pi_1$ see e.g.\ Theorem 12.5 in \cite{billingsley2013}) is again Lipschitz continuous.
Let $K_F$ denote the Lipschitz constant of $F$.
Combining this with \eqref{SecPeq30} we conclude that $d\left( \prod_{i=1}^d\mathbb{L}_{l_i}(\mu_i),F^{-1}([a,\infty)) \right)>\delta/K_F$, hence the second statement is obtained by applying Result \ref{SecPRresult3}.
\end{proof}

\bibliographystyle{abbrv}
\bibliography{bib}

\end{document}